\documentclass[preprint,1p]{elsarticle}

\makeatletter
 \def\ps@pprintTitle{%
 	\let\@oddhead\@empty
 	\let\@evenhead\@empty
 	\def\@oddfoot{\footnotesize\itshape
 		{} \hfill\today}%
 	\let\@evenfoot\@oddfoot
 }
\makeatother
\usepackage{lipsum} 
\ExplSyntaxOn
\NewDocumentEnvironment{multiequation}{b}
 {
  \vantiempham:n { #1 }
 }
 {}

\seq_new:N \l__vantiempham_md_seq

\cs_new_protected:Nn \vantiempham:n
 {
  \seq_set_split:Nnn \l__vantiempham_md_seq { \\ } { #1 }
  $$
  \seq_map_function:NN \l__vantiempham_md_seq \__vantiempham_md_item:n
  $$
 }
\cs_new_protected:Nn \__vantiempham_md_item:n
 {
  \begin{minipage}{\dim_eval:n {\displaywidth/(\seq_count:N \l__vantiempham_md_seq)} }
  \begin{equation}
  \cs_set_eq:Nc \label { ltx@label }
  #1
  \vphantom{\seq_use:Nn \l__vantiempham_md_seq {}}
  \end{equation}
  \end{minipage}
 }

\ExplSyntaxOff  
\usepackage[unicode]{hyperref}
\usepackage{enumitem}
\makeatletter
\providecommand*{\cupdot}{%
  \mathbin{%
    \mathpalette\@cupdot{}%
  }%
}
\newcommand*{\@cupdot}[2]{%
  \ooalign{%
    $\m@th#1\cup$\cr
    \hidewidth$\m@th#1\cdot$\hidewidth
  }%
}
\makeatother
\usepackage{circuitikz}

\usepackage{tikz-cd}
\usepackage{multirow}
\usepackage{url}
\usepackage{pst-node}
\usepackage{comment}

\usepackage[makeroom]{cancel}
 
\usepackage{soul}
\usepackage{tikz}

\usepackage{academicons}
\usepackage{xcolor}

\newcommand{\orcid}[1]{\href{https://orcid.org/#1}{\textcolor[HTML]{A6CE39}{\aiOrcid}}}

\usepackage{latexsym}
\usepackage{indentfirst}
\usepackage{amsxtra}
\usepackage{amssymb}
\usepackage{amsthm}
\usepackage{amsmath}
\usepackage{mathrsfs} 

\usepackage{xcolor}
\usepackage{color}

\usepackage{amsfonts}

\usepackage[capitalise]{cleveref}

\newtheorem{theor}{Theorem}[section]
\newtheorem{prop}[theor]{Proposition}
\newtheorem{lemma}[theor]{Lemma}
\newtheorem{cor}[theor]{Corollary}
\theoremstyle{definition}               
\newtheorem{defin}[theor]{Definition}
\newtheorem{ex}[theor]{Example}
\newtheorem{exs}[theor]{Examples}
\newtheorem{rem}[theor]{Remark}
\newtheorem{rems}[theor]{Remarks}
\newtheorem{que}[theor]{Question}

\newtheorem*{alg*}{Algorithm}
\DeclareMathOperator{\E}{E}
\DeclareMathOperator{\Hom}{Hom}

\newtheorem*{notation*}{Notation}

\DeclareMathOperator{\Sym}{Sym}

\DeclareMathOperator{\Aut}{Aut}
\DeclareMathOperator{\End}{End}

\DeclareMathOperator{\id}{id}

\DeclareMathOperator{\im}{Im}

\usepackage{lastpage}

\makeatletter
\def \@evenhead {\thepage\ of \pageref{LastPage} \hfil \slshape \leftmark } 
\def \@oddhead {{\slshape \rightmark }\hfil \thepage\ of \pageref{LastPage}} 
\makeatother

\begin{document}

\let\today\relax 

\begin{frontmatter}

 	\title{\LARGE{Quasi racks,  quasi bijective and quasi non-degenerate set-theoretic solutions of the Yang–Baxter equation} \tnoteref{mytitlenote}}
	 \author[unile]{Marzia~MAZZOTTA}
  \ead{marzia.mazzotta@unisalento.it (ORCID: 0000-0001-6179-9862)}
	\author[unile]{Paola~STEFANELLI}
	\ead{paola.stefanelli@unisalento.it (ORCID: 0000-0003-3899-3151)}
\author[unibe,uniw]{Magdalena~WIERTEL}
 \ead{m.wiertel@mimuw.edu.pl (ORCID: 0000-0002-0859-6844)}

 \address[unile]{Department of Mathematics and Physics “Ennio De Giorgi”,  University of Salento, Via Provinciale Lecce-Arnesano, 73100 Lecce (Italy)}

\address[unibe]{Department of Mathematics, Vrije Universiteit Brussel, Pleinlaan 2, 1050 Brussel (Belgium)}

 \address[uniw]{Institute of Mathematics, University of Warsaw,
Banacha 2, 02--097 Warsaw (Poland)}
 
\begin{abstract}
This work initiates a systematic study of the class of \emph{quasi bijective} and \emph{quasi non-degenerate} solutions to the set-theoretic Yang--Baxter equation. The motivation stems from the observation that solutions that arise from dual weak braces belong to these classes. The notions of \emph{quasi rack} and \emph{derived solution} are introduced and examined, extending the classical definitions. Additionally, a family of quasi left non-degenerate solutions is described in terms of quasi racks and \emph{g-twists}, analogous to the left non-degenerate case.
Furthermore, we completely characterize a class of quasi racks that are P{\l}onka sum of racks.

    \end{abstract} 
	

\end{frontmatter}

\section*{Introduction}

The \emph{Yang--Baxter equation} was first introduced by Yang \cite{Ya67} in theoretical physics and, independently, by Baxter \cite{Ba72} in statistical mechanics. Since then, the search for its solutions has drawn the attention of many researchers and remains an open problem. Due to the complexity of known solutions, Drinfel'd \cite{Dr92} suggested focusing on a simplified subclass. Namely, if $X$ is a set and
$r:X\times X\to X\times X$ is a map, the pair $(X, r)$ is a \textit{set-theoretic solution of the Yang--Baxter equation}, briefly a \emph{solution}, if the identity
\begin{align}\tag{YBE}\label{YBE}
\left(r\times\id_X\right)
\left(\id_X\times \, r\right)
\left(r\times\id_X\right)
= 
\left(\id_X\times \, r\right)
\left(r\times\id_X\right)
\left(\id_X\times \, r\right)
\end{align}
is satisfied. We will write $r\left(x, y\right) = \left(\lambda_{x}\left(y\right), \rho_{y}\left(x\right)\right)$, where $\lambda_{x}$ and $\rho_{y}$ are maps from $X$ into itself, for all $x,y\in X$. Then
$(X, r)$ is a solution if and only if for all $x,y,z \in X$
    \begin{align}
     &\label{first} \lambda_x\lambda_y(z)=\lambda_{\lambda_x\left(y\right)}\lambda_{\rho_y\left(x\right)}\left(z\right)\tag{Y1}\\
    &  \label{second}\lambda_{\rho_{\lambda_y\left(z\right)}\left(x\right)}\rho_z\left(y\right)=\rho_{\lambda_{\rho_y\left(x\right)}\left(z\right)}\lambda_x\left(y\right)\tag{Y2}\\
      &\label{third}\rho_z\rho_y(x)=\rho_{\rho_z\left(y\right)}\rho_{\lambda_y\left(z\right)}\left(x\right),\tag{Y3}
  \end{align} 
A solution is \textit{bijective} if $r$ is a bijective map and, in particular, \textit{involutive} if $r^2=\id_{X\times X}$; \emph{idempotent} if $r^2=r$; 
\textit{left non-degenerate} if $\lambda_x \in \Sym_X$, for all $x \in X$; 
\textit{right non-degenerate} if $\rho_x \in \Sym_X$, for all $x \in X$, and \emph{non-degenerate} if it is both left and right non-degenerate. Only recently it has been proven in \cite{JePi25} that all non-degenerate solutions are bijective.
\smallskip

Self-distributive structures, \emph{shelves} for short, appeared in several contexts in mathematics since the late 19th century. For example, they proved to be useful in knot theory, see \cite{ElNe74, FeSaFa04, FeRo92} and in Hopf algebras, \cite{ANDRUSKIEWITSCH}. Moreover, structures of this type, especially racks and quandles, play an important role in the study of various aspects of the Yang--Baxter equation, see for example \cite{Le18, LeVe19}. Other similar approaches have been used for several classes of solutions, for example non-degenerate involutive ones in \cite{BoKiStVo21, Rump05} and idempotent left non-degenerate ones in \cite{StVo21}.



Recall that a \emph{(left) shelf} $(X, \triangleright)$ is a set $X$ equipped with a (left) self-distributive binary operation $\triangleright$ such that
\begin{equation*}
    x \triangleright (y \triangleright z)= (x \triangleright y) \triangleright (x \triangleright z)
\end{equation*} 
 for all $x,y,z\in X$. For arbitrarily $x\in X$ we denote by $L_{x}$ the left translation, that is, the map $L_x:X \to X$ such that $y \mapsto x \triangleright y$. If left translations of the shelf $(X, \triangleright)$ are bijective for all $x \in X$, then it is called a \emph{rack}. Moreover, a rack is a \emph{quandle} if additionally $L_x(x) = x$ for all $x \in X$.

It is well known that if $(X, \triangleright)$ is a shelf, then the map 
$r_\triangleright(x, y)=\left(y, y\triangleright x\right)$
    for all $x, y \in X$, is a left non-degenerate solution on $X$.
    Conversely, for arbitrary left non-degenerate solution $(X, r)$ one can define a shelf $(X, \triangleright_r)$, where $
    x \triangleright_r
    y:= \lambda_x\rho_{\lambda_y^{-1}(x)}(y)$ 
    for all $x, y \in X$, whose associated solution is called the \emph{(left) derived solution} of $(X,r)$. Furthermore, any left non-degenerate solution $\left(X, r\right)$ is bijective if and only if $\left(X, \triangleright_r\right)$ is a rack. 
    Moreover, a left non-degenerate solution can be described in terms of a family of automorphisms of the left shelf associated with it, i.e., the \emph{Drinfel'd twist}, corresponding to the maps $\lambda_x$ and of the shelf operation itself, see \cite[Theorem 2.15]{DoRySt24}.
\medskip

In the present paper, we initiate a systematic study of the class of \emph{quasi bijective} and \emph{quasi non-degenerate} solutions. The main motivation for this study comes from the observation that, on the one hand, these are the generalizations of bijective and non-degenerate solutions, respectively, and on the other hand, solutions originating from dual weak braces \cite{CaMaMiSt22} naturally fall within these classes. Let us emphasize that the latter algebraic structures are generalizations of the celebrated class of skew braces, \cite{GuVe17, Ru07}, that are crucial in the study of arbitrary non-degenerate solutions. 
A solution $(X, r)$ is said to be \emph{quasi bijective} if there exists a (unique) solution $\left(X, r^{-}\right)$ such that
      \begin{align*}
          rr^-r=r, \quad r^-rr^-=r^- \quad \text{and} \quad r^{-}r=rr^{-}.
      \end{align*}
Moreover,  $r$ is \emph{quasi left non-degenerate} if
        there exists a (unique) map $\lambda_x^{-}: X \to X$ such that, for all $x,y \in X$,
   \begin{align*}
       \lambda_x\lambda_x^{-}\lambda_x=\lambda_x, \quad \lambda_x^-\lambda_x\lambda_x^-=\lambda_x^-, \quad \lambda_x^0:=\lambda_x\lambda_x^{-}=\lambda_x^-\lambda_x, \quad\text{and}\quad \lambda_x^0\lambda_y=\lambda_y\lambda_x^0.
    \end{align*}
Similarly, $r$ is \emph{quasi right non-degenerate} if $\rho_x$ admits the above conditions for all $x\in X$.
In addition, $r$ is \emph{quasi non-degenerate} if it is both quasi right and left non-degenerate. 
A natural direction that emerges at this point is to understand these solutions through an analogue of the derived solutions and shelves, in the spirit of the fruitful approach developed for left non-degenerate solutions. It is possible and potentially useful to investigate this class, because although such solutions are not necessarily bijective, certain properties of the defined inverses allow us to extend the theory to this context. Following further this idea, in \cref{sec: quasiracks} we introduce the class of quasi racks, that are shelves whose left multiplications have semigroup inverses (in the underlying transformation monoid) satisfying certain additional properties, similar to $\lambda$ maps of quasi left non-degenerate solutions. 

 A quasi rack~$(X,\triangleright)$ is a \emph{quasi quandle} if left translation maps satisfy $L_x(x)=x$, for all~\hbox{$x\in X$}. In parallel, we develop a generalized notion of derived solutions that provides a deeper understanding of these almost-bijective solutions. 
    To this end, we present three families of quasi racks $(X, \triangleright)$ that give rise to a quasi non-degenerate solution defined as $r_\triangleright(x, y)=\left(L_x^0(y), L_y(x)\right)$. We call this solution the \emph{derived solution} of a shelf. These families are characterized by one of the following properties for all $x, y \in X$: 
\begin{align}\label{ast}
     L^0_{L_x(y)} &= L^0_xL^0_y \tag{$\ast$}\\
         L_y(x)&=L_{L^0_x(y)}(x)  \tag{$\ast\ast$}\label{astast} \\
          L^0_x(x)&=x. \tag{$\ast\ast\ast$} \label{astastast}
     \end{align}
In particular, condition \eqref{astastast} implies \eqref{astast} and also guarantees that the derived solution is quasi bijective. Nevertheless, no other implications hold in general. Moreover in \cref{table: enumeration} we give (using computer) the numbers, up to isomorphism, of all defined structures related to quasi racks of size up to $4$.
Let us also note that the proposed conditions are not enough to identify all quasi racks that yield a derived solution.\\       
As we observe in \cref{sec:Plonka}, quasi racks satisfying \eqref{ast} and \eqref{astastast} can be realized as P{\l}onka sums of racks, consistently with the notion in \cite{Plo67, PloRom92}. Consequently, the derived solution associated to the quasi rack $(X, \triangleright)$ is the strong semilattice (cf.~\cite{CaCoSt21}, which is equivalent to Płonka sums of solutions, cf.~\cite{St25}) of the derived solutions that arise from each rack.\\
In addition, in \cref{Section3} we identify a family of quasi left non-degenerate solutions for which the structure $(X, \triangleright_r)$ is a shelf, where $\triangleright_r$ is the usual binary operation mentioned before. Consistently with \cite[Theorem 2.15]{DoRySt24}, these quasi left non-degenerate solutions can be described in terms of what we call \emph{g-twists}, notions that generalize the known ones of Drinfel'd twists, as we show in \cref{sec:twists}. 


\medskip

\section{Preliminaries}\label{sec:preliminaries}
In this section, we provide some preliminaries that are useful throughout the paper and which, for convenience, we subdivide into two subsections.

\smallskip
\subsection{Basics on weak braces}
To deal with the algebraic structures of weak brace, we introduce  the necessary background on semigroup theory. For details, we refer the reader to some classical books such as \cite{ClPr61, La98, PeRe99}.

A \emph{regular semigroup} is a semigroup $S$ in which every element is regular, i.e., for any $a \in S$ there exists an element $x$ in $S$, called \emph{an inverse} of $a$, such that $axa = a$. In particular, $ax$ is an idempotent of $S$. 
A semigroup $S$ is said to be \emph{completely regular semigroup} if for every $a \in S$ there exists $x \in S$ such that $axa=a$, $xax=x$, and $ax=xa$. In his early work, Clifford \cite{Cl41} called completely regular semigroups \emph{semigroups admitting relative inverses}. A regular semigroup $S$ is called an \emph{inverse semigroup}, if every element $a \in S$ admits a unique inverse. Roughly speaking, such inverses behave similarly to inverses in a group. For example, in any inverse semigroup $S$ we have that  $(xy)^{-}=y^{-} x^{-}$ and $(x^{-})^{-}=x$, for all $x,y \in S$.
An inverse semigroup $S$ is called a \emph{Clifford semigroup} if all idempotents are central, or equivalently, $xx^{-}=x^{-}x$, for all $x \in S$. Clifford semigroups can be also characterized as \emph{strong semilattices} of groups.

\medskip
Let us now recall the definition contained in \cite[Definition 5]{CaMaMiSt22}.
\begin{defin} 
    Let $S$ be a set endowed with two operations $+$ and $\circ$ such that $\left(S,+\right)$ and $\left(S,\circ\right)$ are inverse semigroups. Then, $(S, +, \circ)$ is a \emph{weak (left) brace} if
 \begin{align*}
    x\circ\left(y + z\right)
    = x\circ y -x+ x\circ z\qquad \text{and}\qquad x\circ x^-=-x+x,
 \end{align*}
for all $x,y,z\in S$, where $-x$ and $x^-$ denote the inverses of $x$ with respect to $+$ and $\circ$, respectively.  
\end{defin}

Semigroups $(S, +)$ and $(S, \circ)$ are called the additive and multiplicative semigroup of a weak brace $(S, +, \circ)$, respectively. Clearly, the sets of the idempotents $\E(S,+)$ and $\E(S,\circ)$ coincide, so we simply denote this set by $\E(S)$. In addition, by \cite[Lemma 1.4]{CaMaSt24}, 
\begin{align*}
   \forall\, x\in S, \ e \in \E(S) \qquad  e \circ x=e+x.
\end{align*}
Obviously, if $|\E(S)|=1$, then $(S,+, \circ)$ is a skew brace \cite{GuVe17}. 
Furthermore, in \cite[Theorem 8]{CaMaMiSt22} it is proved that the additive semigroup of any weak brace is necessarily Clifford. A weak brace $(S, +, \circ)$ will be called a \emph{dual weak brace} if also $(S, \circ)$ is Clifford.  \\
Any Clifford semigroup $\left(S, \circ\right)$ determines two trivial dual weak braces, by setting $x + y:= x\circ y$ or $x + y:= y\circ x$, for all $x,y\in S$. In general, if $(S, +, \circ)$ is a weak brace, then the structure $S^{op}:=(S, +^{op}, \circ)$ also is a weak brace, where $x+^{op} y:= y+x$, for all $x, y \in S$. Furthermore, by \cite[Theorem 1]{CaMaSt24}, any dual weak brace can be obtained as a \emph{strong semilattice of skew braces}. 




\medskip

For arbitrary set $S$ we denote by $S^S$ the monoid of all functions $S\rightarrow S$.
As usual, if $(S,+, \circ)$ is a weak brace, consider the maps $\lambda: S\to \End(S,+), \,x\mapsto\lambda_x$ and $\rho: S\to S^S, \,y\mapsto\rho_y$ defined by 
\begin{align*}
   \lambda_x\left(y\right) = -x+x\circ y\qquad \text{and} \qquad \rho_y\left(x\right) = \lambda_x\left(y\right)^{-}\circ x \circ y,
\end{align*}
for all $x,y\in S$, respectively. We have $\lambda_x(y)=x\circ \left(x^- +y \right)$ and $\rho_y(x)=\left(x^-+y\right)^- \circ y$, for all $x,y \in S$. Besides, by \cite[Proposition 7]{CaMaMiSt22}, the map $\lambda$ is a homomorphism of $\left(S,\circ\right)$ into the endomorphisms semigroup of $\left(S,+\right)$ and, by \cite[Proposition 10]{CaMaMiSt22}, the map $\rho$  is an anti-homomorphism of $\left(S,\circ\right)$ into the monoid $S^S$.  For further properties, we refer the reader to \cite{CaMaMiSt22}.

Any weak brace gives rise to a solution, as we recall in the following (see \cite[Theorem 11, Theorem 13]{CaMaMiSt22}).

\begin{theor}\label{teo_quasi_bij}
    Let $S$ be a weak brace and consider the map $r:S\times S\to S\times S$ defined by $r\left(a,b\right)
    = \left(\lambda_a(b), \ \rho_b(a)\right)$, 
for all $a,b\in S$. Then $(S, r)$ is a solution, called the \emph{solution associated with $S$}. Moreover, the following hold:
\begin{align*}
      r\, r^{op}\, r = r, \qquad
      r^{op}\, r\, r^{op} = r^{op}, \qquad \text{and}\qquad rr^{op} = r^{op}r,
  \end{align*}
  where $r^{op}$ is the solution associated with the opposite weak brace $S^{op}$ of $S$.
\end{theor}

The result in \cref{teo_quasi_bij} establishes that the solution $r$ associated with a dual weak brace is close to being bijective.  In particular, if $S$ is a skew brace, the solution $r$ is non-degenerate and bijective with $r^{-1} = r^{op}$. 
Similarly, the next result shows that the solution coming from a dual weak brace is also close to being non-degenerate.

\begin{prop}\label{unique_lam_rho}
    Let $(S, +, \circ)$ be a dual weak brace. Then the sets
    \begin{align*}
        \Lambda(S)=\{\lambda_x \, \mid \, x \in S\}
\qquad \text{and} \qquad \varrho (S)=\{\rho_x \, \mid\,  x \in S\}   \end{align*}
    are Clifford subsemigroups of the monoid $S^S$.
    \begin{proof}
        Let $x, y \in S$. Then, by \cite[Proposition 7]{CaMaMiSt22}, $\lambda_x\lambda_y=\lambda_{x \circ y} \in  \Lambda(S)$ and so $\Lambda(S)$ is a subsemigroup of $S^S$. Moreover, $\lambda_x\lambda_{x^-}\lambda_x=\lambda_{x \circ x^-\circ x}=\lambda_x$. In addition, we have that $\E(\Lambda(S))=\{\lambda_e \, \mid \, e\in \E(S)\}$. In fact, if $\lambda_a \in \E(\Lambda(S))$, then $\lambda_{a^-}\lambda_a\lambda_a=\lambda_{a^-}\lambda_a$ which implies $\lambda_a=\lambda_{a \circ a^-}$. 
        Moreover, for all $a \in S$,
        \begin{align*}
            \lambda_x\lambda_{x^-}\lambda_a=\lambda_{x \circ x^- \circ a}=\lambda_{a \circ x \circ x^-}=\lambda_a\lambda_{x}\lambda_{x^-}
        \end{align*}
Similarly, one can prove the claim for the set $\varrho (S)$ by using \cite[Proposition 10]{CaMaMiSt22}.
    \end{proof}
\end{prop}

\smallskip

From Proposition 2.7 in \cite{GRAY20084801} we get the following immediate consequence.

\begin{cor}\label{cor:Gray-Mitchell}
If $(S, +, \circ)$ is a finite dual weak brace, then $\lambda_{x_{\mid\im(\lambda_x)}}$ and $\rho_{x_{\mid\im(\rho_x)}}$ are bijective for all $x\in S$. 

\end{cor}

\medskip

\subsection{Basics on shelves}

We recall some basics on the algebraic structures of shelves and racks. For further details, see for instance \cite{FeSaFa04, FeRo92,  Le18, LeVe19}.

\smallskip

\begin{defin}
A \emph{(left) shelf} $(X, \triangleright)$ is a set $X$ equipped with a left self-distributive binary operation~$\triangleright$ such that
\begin{equation}\label{self_distri}
    x \triangleright (y \triangleright z)= (x \triangleright y) \triangleright (x \triangleright z), 
\end{equation} 
 for all $x,y,z\in X$.
\end{defin}
If we denote with $L_x:X \to X, y \mapsto x \triangleright y$ the operator of left multiplication by $x$, for all $x \in X$, then equation \eqref{self_distri} is equivalent to require that $L_x$ be a shelf homomorphism. Clearly, one can define right racks and denote by $R_x$ the right multiplications.

\begin{defin}
A shelf $(X, \triangleright)$ is a \emph{rack} if all maps $L_x$ are bijective. Moreover, a rack $(X, \triangleright)$ is a \emph{quandle} if, in addition, $L_x(x) = x$, for all $x \in X$.
\end{defin}

\smallskip

The following known results highlight important connections between shelves and left non-degenerate solutions.

\begin{prop}\label{YBE_and_shelf}
    Let $(X, \triangleright)$ be a shelf.
    Then, the map 
$r_\triangleright(x, y)=\left(y, y\triangleright x\right)$,
    for all $x, y \in X$, is a left non-degenerate solution on $X$.
    Conversely, if $(X, r)$ is a left non-degenerate solution, set
    \begin{align*}
    x \triangleright_r
    y:= \lambda_x\rho_{\lambda_y^{-1}(x)}(y),
     \end{align*}
    for all $x, y \in X$, then the structure $(X, \triangleright_r)$ is a shelf.
\end{prop}
We recall that the solution associated with the shelf $\left(X, \triangleright_r\right)$ is called \emph{(left) derived solution} of $(X,r)$.

\smallskip

\begin{theor}[Definition 2.14 - Theorem 2.15, \cite{DoRySt24}]\label{th-lns}
Let $\left(X,\,\triangleright\right)$ be a shelf and consider a map $\varphi:X \to \mathrm{Sym}_X, x\mapsto \varphi_x$. Then, the map 
$r_\varphi:X\times X\to X\times X$ defined by
    \begin{align*}
        r_\varphi\left(x, y\right)  
        = \big(\varphi_x\left(y\right),\, \varphi^{-1}_{\varphi_x\left(y\right)}\left(\varphi_x\left(y\right)\triangleright x
\right)\big),  
    \end{align*}
    for all $x,y\in X$, is a solution if and only if $\varphi$ is a \emph{twist}, namely, $\varphi_x \in \Aut(X, \triangleright)$ and
    \begin{align*}
        \varphi_x\varphi_y = \varphi_{\varphi_x\left(y\right)}\varphi_{\varphi^{-1}_{\varphi_x\left(y\right)}L_{\varphi_x\left(y\right)}\left(x\right)},
    \end{align*}
    for all $x, y \in X$.
Moreover, any left non-degenerate solution can be obtained in this way.
\end{theor}

Let us also recall from \cite{DoRySt24} one of the 
consequences of this theorem.
\begin{prop}\label{bijective=rack}
    Any left non-degenerate solution $\left(X, r\right)$ is bijective if and only if $\left(X, \triangleright_r\right)$ is a rack. 
\end{prop}

Moreover, we also recall that, if $X$ is a finite set, a left non-degenerate solution $\left(X, r\right)$ is right non-degenerate and bijective if and only if $\left(X, \triangleright_r\right)$ is a rack. For more details, we refer the reader to Corollaries 2.16 and 2.17 in \cite{DoRySt24}.

 \medskip

\section{Quasi racks and quasi bijective non-degenerate solutions}\label{sec: quasiracks}
In this section, we introduce the notions of quasi racks and quasi non-degenerate solutions, with the aim to generalize the notions of the previous section. In particular, we identify three families of quasi racks that produce a quasi non-degenerate derived solution.

\medskip

Let us introduce the following definition.

\begin{defin}\label{defn-quasi}
      A solution $(X, r)$ is said to be \emph{quasi bijective} if there exists a (unique) solution $\left(X, r^{-}\right)$ such that
      \begin{align*}
          rr^-r=r, \quad r^-rr^-=r^- \quad \text{and} \quad r^{-}r=rr^{-}.
      \end{align*}
\end{defin}

\begin{rem}\label{r-}
If $(X, r)$ is a solution which admits a relative inverse $r^{-}$, then $(X, r^{-})$ is also a solution. Indeed, it is straightforward to check that 
$(r\times \id_{X})^{-} = r^{-}\times \id_{X}$ and~$( \id_{X}\times r)^{-} = \id_{X}\times r^{-}$. Moreover, observe that $(ab)^{-} = b^{-}a^{-}$, for arbitrary $a, b$ with relative inverses $a^{-}$ and $b^{-}$, respectively. Taking the inverses of both sides of the \eqref{YBE}, 
we thus get $(r^{-}\times \id_{X})(\id_{X} \times r^{-})(r^{-}\times \id_{X}) = (\id_{X} \times r^{-})(r^{-}\times \id_{X})(\id_{X} \times r^{-})$ and the assertion follows. 
\end{rem}

\smallskip

\begin{defin}\label{defn-left-right-quasi}
Let $(X, r)$ be a solution. Then $r$ is said to be:
    \begin{enumerate}
        \item[1)] \emph{quasi left non-degenerate} if
        there exists a (unique) map $\lambda_x^{-}: X \to X$ such that, for all $x,y \in X$,
   \begin{align*}
       \lambda_x\lambda_x^{-}\lambda_x=\lambda_x, \quad \lambda_x^-\lambda_x\lambda_x^-=\lambda_x^-, \quad \lambda_x^0:=\lambda_x\lambda_x^{-}=\lambda_x^-\lambda_x, \quad\text{and}\quad \lambda_x^0\lambda_y=\lambda_y\lambda_x^0;
    \end{align*}
    \item[2)] \emph{quasi right non-degenerate} if 
    there exists a (unique) map $\rho_x^{-}: X \to X$ such that, for all $x,y \in X$,
    \begin{align*}
       \rho_x\rho_x^{-}\rho_x=\rho_x, \quad \rho_x^-\rho_x\rho_x^-=\rho_x^-, \quad  \rho_x^0:=\rho_x\rho_x^{-}=\rho_x^-\rho_x, \quad\text{and}\quad \rho_x^0\rho_y=\rho_y\rho_x^0;
    \end{align*}
    \item[3)] \emph{quasi non-degenerate} if $r$ is both quasi right and left non-degenerate.
    \end{enumerate}
\end{defin}

\smallskip

\begin{rem}Observe that the condition $\lambda_{x}^{0}\lambda_{y} = \lambda_{y}\lambda_{x}^{0}$ for arbitrary $x, y\in X$ does not follow from the existence of the inverse in the sense of the \cref{defn-left-right-quasi}. Indeed, consider a set $X$ with $|X|\geq 1$ and let $(X, r)$ be the solution given by $r(x,y)=(x, x)$, for all $x, y \in X$. Then the maps $\lambda_x$ have the inverses in the sense of our definition, but, if $x \neq y$ it is not true that $\lambda_x^0\lambda_y \neq \lambda_y\lambda_x^0$.
\end{rem}

\smallskip

 We will refer to maps $r^{-}$, $\lambda_{x^-}$, and $\rho_{x^-}$ as \emph{relative inverses} of maps $r$, $\lambda_x$, and $\rho_x$, respectively, for all $x \in X$. Observe that these relative inverses, if they exist, are all unique 
 (see  \cite[Corollary~4, p.~275]{Miller-Clifford}). For the sake of completeness, we outline the proof.

 \begin{lemma}\label{uniqueness}
For arbitrary semigroup $S$ and $a\in S$, if there exists $x \in S$ such that $axa=a$, $xax = x$ and $xa = ax$, then $x$ is uniquely determined. 
 \end{lemma}
\begin{proof}
  Let $s$ be arbitrary element of a semigroup $S$ such that there exists $x \in S$ with $sxs= s$, $xsx=x$ and $xs=sx$.
  Then $a$, $x$ and $ax$ are $\mathcal{H}$-related, i.e. $aS^{1} = axS^{1} = xS^{1}$ and $S^{1}a = S^{1}ax = S^{1}x$. It follows that $\mathcal{H}_{a} = \mathcal{H}_{x} = \mathcal{H}_{ax}$, where $\mathcal{H}_{y}$ denotes the $\mathcal{H}$-class in $S$ containing $y$. In particular, this $\mathcal{H}$-class contains an idempotent $ax$, so from Green's theorem (\cite[Theorem 2.16]{ClPr61}) it is a subgroup of $S$. Thus if $x' \in S$ is other element such that $ax'a = a$, $x'ax' = x'$ and $ax' = x'a$, we have $
ax = ax'= e$, as this is a unique idempotent in the subgroup $\mathcal{H}_{a}$. It follows that $x' = e x' = xax'ax = xax = x$, and the assertion follows.
\end{proof}

From the uniqueness of relative inverses it follows that for quasi left (right, respectively) non-degenerate solutions we have that $\lambda_{x}^{0}\lambda_{y}^{-} = \lambda_{y}^{-} \lambda_{x}^{0}$ ($\rho_{x}^{0}\rho_{y}^{-} = \rho_{y}^{-} \rho_{x}^{0}$, respectively) and therefore also $\lambda_{x}^{0}\lambda_{y}^{0} = \lambda_{y}^{0}\lambda_{x}^{0}$ ($\rho_{x}^{0}\rho_{y}^{0} = \rho_{y}^{0} \rho_{x}^{0}$, respectively) for all $x, y\in X$.

\smallskip

\begin{ex}
    The solution $(S, r)$ associated with any weak brace $S$ is quasi bijective. Moreover, if $S$ is dual, the solution $r$ also is quasi non-degenerate.
\end{ex}


\smallskip

The following is an example of a quasi non-degenerate solution that can not be obtained from a dual weak brace if $f, g \neq \id_X$.

\begin{ex}\label{ex:quasi-Lyubashenko}
  Let $X$ be a set and $f, g: X \to X$ two commuting maps such that they admit as relative inverses $f^-$ and $g^-$, respectively, in $X^X$. Then $r(x, y)=(f(y), g(x))$ is a quasi non-degenerate solution on $X$. In particular, if $g=f^-$, then the solution $r(x, y)=\left(f(y), f^-(x)\right)$ is cubic, i.e., $r^3=r$. In particular, if $r$ is also bijective, we obtain involutive Lyubashenko solutions.
\end{ex}

Motivated by the connection between non-degenerate solutions 
and self-distributive structures, we introduce the following definition.

\begin{defin}\label{def:quasi-rack}
    A shelf $(X,\triangleright)$ is called \emph{quasi rack} if, for all $x \in X$, the left multiplication $L_x$ admits a relative inverse in $X^X$, denoted by $L_x^-$, such that
    \begin{align*}
        L_xL_x^-L_x=L_x, \quad L_x^-L_xL_x^-=L_x^-, \quad L^0_x:=L_xL_x^-=L_x^-L_x,    \quad \text{and} \quad L^0_xL_y=L_yL^0_x, 
    \end{align*}
    for all $x, y \in X$.
In addition, a quasi rack $(X,\triangleright)$ is a \emph{quasi quandle} if $L_x(x)=x$, for all $x \in X$.
    \end{defin}

The relative inverse $L_{x}^{-}$, if exists, by \cref{uniqueness}, is uniquely determined. It follows that $(L_{x}L_{y})^{0} = L_{x}^{0}L_{y}^{0}$ for arbitrary $x, y$. We will use this fact in the forthcoming proofs without further comments.

\begin{rem} 
Note that there exist shelves that satisfy all conditions from \cref{def:quasi-rack}, except the last one. Indeed, let us consider a set $X$ with $|X|\geq 1$ endowed with the  shelf structure given by $x\triangleright y = x$, for all $x,y\in X$. Then every $L_x$ admits a relative inverse ($L_{x}^{-} = L_{x}$), but the condition $L^0_xL_y=L_yL^0_x$ is not satisfied for $x\neq y$. 
\end{rem}

\smallskip

\begin{ex}\label{ex_idemp}
    Let $X$ be a set and $f: X \to X$ an idempotent map. Set $x \triangleright y:=f(y)$, for all $x, y \in X$, then $(X, \triangleright)$ is a quasi rack.
\end{ex}

\begin{exs}\label{exs_quasi} Let $X$ be a Clifford semigroup.
 \begin{enumerate}
     \item  Define $x \triangleright y:= x^{-} y  x$, for all $x, y \in X$. Then $(X, \triangleright)$ is a quasi quandle that we call \emph{conjugation quasi quandle}.
     Clearly, if $X$ is a group, $(X, \triangleright)$ is the \emph{conjugation quandle}.
     
   \item More generally, let $X$ be a Clifford semigroup, $e \in \E(X)$, and set $x \triangleright y:=x^{-}yxe$, for all $x, y \in X$. Then $(X, \triangleright)$ is a quasi rack with $L^-_x=L_{x^-}$, for all $x \in X$.
   
     \item Define $x \triangleright y:= xy^{-} x$, for all $x, y \in X$. Then $(X, \triangleright)$ is a quasi quandle that we call \emph{core quasi quandle}.   
     Clearly, if $X$ is a group, $(X, \triangleright)$ is the \emph{core quandle}. 
     \end{enumerate}  
\end{exs}

\smallskip

The following example of a quasi rack is inspired by the construction of racks provided in \cite{ANDRUSKIEWITSCH}.
\begin{ex}
Let $(X, \triangleright)$ be a rack, $S$ a finite set, and $\alpha: X\times X\rightarrow S^{\,S\times S}$ a function. We will write $\alpha_{i, j}(s, t)= \alpha_{i, j}(s)(t)$, where $\alpha_{i, j}(s):S \rightarrow S$ is an element of the semigroup $S^{S}$. Assume that, for all $i, j, k \in X$
 and $s, t,u \in S$
 \begin{enumerate}
    \item the map $\alpha_{i, j}(s): S\rightarrow S$ has commuting relative inverse; 
    \item $\alpha_{i, j\triangleright k}(s)\alpha_{j, k}(t) = \alpha_{i\triangleright j, i\triangleright k}(\alpha_{i, j}(s, t))\alpha_{i, k}(s)$;
    \item $\alpha_{j, k}(t, (\alpha_{i, k}(s))^{0}(u)) = (\alpha_{i, j\triangleright k}(s))^{0}(\alpha_{j, k}(t, u))$.
    \end{enumerate}
Then $X\times S$ is a quasi rack with 
$$
L_{(i, s)}(j, t) = \left(i\triangleright j, \alpha_{i, j}(s)(t)\right),
$$
for all $i, j \in X$, $s, t \in S$. Direct computation shows that $L_{i, s}^{0}(j, t) = (j, \left(\alpha_{i, j}(s))^{0}(t)\right)$.


Taking $\alpha_{i, j}(s) = L_s$, where $(S, \triangleright_{s})$ is a quasi rack, we obtain a direct product of rack $(X, \triangleright)$ and a quasi rack $(S, \triangleright_{S})$. 
\end{ex}

\smallskip


The following technical lemma will be useful for the proof of main theorems of this section.
\begin{lemma}\label{lemma_Lx_gene}
    Let $(X, \triangleright)$ be a quasi rack. Then the following hold:
    \begin{enumerate}
    \item $L^0_xL_y=L^0_xL_{L^0_x(y)}$,
    \item $L_xL_y = L_{L^0_y(x)}L_y$,
    \item $L^0_xL_y = L^0_{L_y(x)}L_y$,
   \item $L^0_xL^0_{L^-_x(y)}=L_x^0L^0_y$,
    \end{enumerate}
    for all $x, y \in X$.
    \begin{proof}
        Let $x, y \in X$. Then $L_xL_y\underset{\eqref{self_distri}}{=}L_{L_x(y)}L_x= L_{L_x\left(L_x^0(y)\right)}L_x\underset{\eqref{self_distri}}{=}L_xL_{L^0_x(y)}$.
      Thus,
        \begin{align*}
            L_x^0L_y=L_x^-L_xL_y{=}L_x^-L_xL_{L^0_x(y)}=L^0_xL_{L^0_x(y)}.
        \end{align*}
        Moreover, as $L_{y}^{0}$ commutes with all $L_{z}$,
$$L_{L^0_y(x)}L_y=L^0_yL_{L^0_y(x)}L_y\underset{\ref{lemma_Lx_gene} - 1}{=}L^0_yL_xL_y=L_xL_y$$ 
        and, similarly, $$L^0_{L_y(x)}L_y=\left(L_{L_y(x)}L_y\right)^0L_y \underset{\eqref{self_distri}}{=}\left(L_xL_y\right)^0L_y=L_x^0L_y.$$
Finally, by $2.$,
      \begin{align*}
          L^-_xL_yL_x=L_x^-L_{L^0_x(y)}L_x = L_x^-L_{L_xL_x^-(y)}L_x \underset{\eqref{self_distri}}{=}L^0_xL_{L^-_x(y)}
      \end{align*}
      which implies $L^0_xL^0_{L^-_x(y)}=L_x^0L^0_y$.
\end{proof}
\end{lemma}




\smallskip

\begin{defin}\label{def:derived_map}
    Let $(X, \triangleright)$ be a quasi rack. Then we call the map $r_{\triangleright}:X\times X\to X\times X$ defined by
    \begin{align}\label{r_triang}
    r_\triangleright(x, y)=\left(L_x^0(y), L_y(x)\right),
\end{align}
for all $x, y\in X$, the \emph{derived map associated to $(X, \triangleright)$}.
\end{defin}

\smallskip

The following theorems provide sufficient conditions for the derived map to be a quasi non-degenerate solution. The conditions we require involve only idempotents and are therefore trivial in the bijective case.

\begin{theor}\label{teo-request0}
        Let $(X, \triangleright)$ be a quasi rack and assume that, for all $x, y\in X$,
 \begin{align}\label{request0}
        L^0_{L_x(y)} = L^0_x L^0_y \tag{$\ast$}
    \end{align}
    Then the derived map $r_\triangleright$ associated to $(X, \triangleright)$ in \eqref{r_triang} is a quasi non-degenerate solution.
\end{theor}
\begin{proof} 
   Initially, for all $x, y\in X$, let $\lambda_x(y):=L_x^0(y)$ and $\rho_y(x):=L_y(x)$. 
     First, we show that, for all $x, y \in X$,
    \begin{align}\label{cons0}
        L_{L^0_x(y)} = L^0_xL_y.
    \end{align}
    In fact, if $x, y \in X$,
    $$
    L_{L^0_x(y)}=L_{L^0_x(y)}L^0_{L_x\left(L^-_x(y)\right)}=L_{L^0_x(y)}L^0_xL^0_{L^-_x(y)}\underset{\ref{lemma_Lx_gene}-4.}{=}L_{L^0_x(y)}L^0_xL^0_y\underset{\ref{lemma_Lx_gene}-1.}{=}L^0_xL_y.
    $$
    Now, if $x, y, z \in X$, we have
\begin{align*}
    \rho_{\rho_y\left(x\right)}\rho_{\lambda_x\left(y\right)}
    &\underset{\eqref{cons0}}{=}
        L_{L_y(x)}L_yL_x^0
        \underset{\eqref{self_distri}}{=}
        L_y L_x L_x^0
        = L_y L_x = \rho_y\rho_x,
    \end{align*}
    namely, equation \eqref{third} is satisfied. It follows that
$$\lambda_{x}\lambda_{y} = \rho_x^{0} \rho_y^{0} = (\rho_y\rho_x)^{0} \underset{\eqref{third}}{=}(\rho_{\rho_y\left(x\right)}\rho_{\lambda_x\left(y\right)})^{0} = 
  \rho_{\lambda_x\left(y\right)} ^{0} \rho_{\rho_y\left(x\right)}^{0} =
 \lambda_{\lambda_x\left(y\right)} \lambda_{\rho_y\left(x\right)},$$ so \eqref{first} follows from \eqref{third}. Furthermore, \eqref{second} follows by
    \begin{align*}
    \lambda_{\rho_{\lambda_y\left(z\right)}\left(x\right)}\rho_z\left(y\right)
    &\underset{\eqref{request0}}=L^0_{L^0_y(z)}L^0_xL_z(y)\underset{\eqref{cons0}}{=}L^0_xL^0_yL_z(y) = \left(L^0_yL^0_x\right)L^0_xL_z(y)\underset{\eqref{request0}}{=}L^0_{L_y(x)}L^0_xL_z(y)\\
    &\underset{\eqref{cons0}}{=}L_{L^0_{L_y(x)}(z)}L^0_x(y)=\rho_{\lambda_{\rho_y\left(x\right)}\left(z\right)}\lambda_x\left(y\right).
\end{align*}
 Finally, note that $\left(X, r_\triangleright\right)$ is clearly quasi non-degenerate.  Therefore, the claim follows.
\end{proof}

\smallskip

\begin{theor}\label{teo-request}
        Let $(X, \triangleright)$ be a quasi rack and assume that, for all $x, y\in X$,
 \begin{align}\label{request1}
        L_y(x) = L_{L^0_x(y)}(x). \tag{$\ast\ast$}
    \end{align}
    Then the derived map $r_\triangleright$ associated to $(X, \triangleright)$ in \eqref{r_triang} is a quasi non-degenerate solution.
\end{theor}
\begin{proof}
 First, for all $x, y\in X$, let $\lambda_x(y):=L_x^0(y)$ and $\rho_y(x):=L_y(x)$.  Then, for all $x, y \in X$,
    \begin{align*}
    \rho_{\rho_y\left(x\right)}\rho_{\lambda_x\left(y\right)}
        &\underset{\eqref{request1}}{=}L_{L_{L^0_x(y)}(x)}L_{L^0_x(y)}
        \underset{\eqref{self_distri}}{=}L_{L^0_x(y)}L_x
        \underset{\ref{lemma_Lx_gene}-2.}{=} L_yL_x
        = \rho_y\rho_x,
    \end{align*}
    namely, equation \eqref{third} holds. Hence, again \eqref{first} follows by \eqref{third} as in proof of \cref{teo-request0}.
    Moreover, if $x,y, z\in X$, we obtain that
    \begin{align*}
    \lambda_{\rho_{\lambda_y\left(z\right)}\left(x\right)}\rho_z\left(y\right)
    &\underset{\eqref{request1}}{=} 
    L^0_{L_{L^0_y(z)}(x)}L_{L^0_y(z)}(y)
    \underset{\ref{lemma_Lx_gene}-3.}{=} L^0_x L_{L^0_y(z)}(y)
    \underset{\eqref{request1}}{=}
L^0_xL_z(y)
\end{align*}
and
\begin{align*}
\rho_{\lambda_{\rho_y\left(x\right)}\left(z\right)}\lambda_x\left(y\right)
&= L_{L^0_{L_y(x)}(z)}L^0_x(y)
\underset{\eqref{request1}}{=}  L_{L^{0}_{L^0_x(y)}L^0_{L_y(x)}(z)}L^0_x(y)\\
&\underset{\eqref{request1}}{=}
L_{L^{0}_{L^0_x(y)}L^0_{L_{L^0_x(y)}(x)}(z)}L^0_x(y)
\underset{\eqref{self_distri}}{=}
L_{L^0_{L^0_x(y)}L^0_x(z)}L^0_x(y)\\
&\underset{\eqref{request1}}{=}L_{L^0_x(z)}L^0_x(y)\underset{\ref{lemma_Lx_gene}-1.}{=}L^0_xL_z(y)
\end{align*}
i.e., \eqref{second} holds.  Therefore, the claim follows.
\end{proof}

\smallskip

\begin{cor}\label{cor-request}
    Let $(X, \triangleright)$ be a quasi rack and assume that for all $x \in X$
     \begin{align}\label{L0}
       L^0_x(x) = x. \tag{$\ast\ast\ast$}
    \end{align}
    Then the derived map $r_\triangleright$ associated to $(X, \triangleright)$ is a quasi non-degenerate solution.
    \begin{proof}
It is enough to observe that the condition \eqref{request1} holds. In fact, for all $x, y \in X$, $L_y(x) = L_x^0L_y(x)\underset{\ref{lemma_Lx_gene}-1.}{=}L^0_xL_{L^0_x(y)}(x) = L_{L^0_x(y)}(x)$.     By \cref{teo-request} the claim follows. 
    \end{proof}
\end{cor}

\smallskip

\begin{rems} \label{rem-ex-astastast-noast} 
From the proof of \cref{cor-request}, it is clear that \eqref{astastast} imply \eqref{astast}. Below, we clarify through a few examples that the other conditions found above do not, in general, imply one another.
\begin{enumerate}
    \item Note that in general conditions \eqref{request0} or \eqref{request1} do not imply condition \eqref{L0}. Indeed, consider the quasi rack $(X, \triangleright)$ from \cref{exs_quasi}-2. Then \eqref{request0} and \eqref{request1} are satisfied, but $L^0_x(x) =xe \neq x$. 
    \item Let $X= \{1, 2, 3\}$ and consider the quasi rack such that $L_{1}(x) = 1$, for all $x\in X$, $L_{2} = L_{1}$ and $L_{3} = \id_{X}$. Then it is easy to check that the condition \eqref{astast} is not satisfied, whereas \eqref{ast} holds.
    \item Let $X = \{1, 2, 3\}$ and consider the quasi rack such that $L_1$ is the map defined by $L_{1}(1) = L_{1}(2) = 1$ and $L_{1}(3) =3$, and additionally $L_{2} = L_{3} = \id_{X}$. Then,  $L_{x}^{-} = L_{x}$, for all $x\in X$, and $L_{1} = L_{1}L_{3}\neq L_{L_{1}(3)} = L_{3}$, so \eqref{ast} is not satisfied, but \eqref{astastast} holds. This example also shows that \eqref{astast} does not imply \eqref{ast}.
\end{enumerate}
\end{rems}

\smallskip

For completeness, we note that there exist quasi racks that do not satisfy any of the above conditions \eqref{ast}, \eqref{astast}, or \eqref{astastast}, yet the associated derived map is a solution.
\begin{rem}
For $X = \{1, 2, 3, 4\}$ let us define the quasi rack such that $L_{1} = L_{4}\equiv 1$, $L_{2} = \id_{X}$, $L_{3}(k) = k$ for $k=1, 2, 3$ and $L_{3}(4) = 1$. Then it can be easily checked that the derived map form \cref{def:derived_map} is a solution, but conditions \eqref{ast} and \eqref{astast} are not satisfied.
\end{rem}


\smallskip

\begin{rem}
Observe that condition \eqref{request1} ensures that maps $\lambda_x, \rho_x$ that define the solution in \cref{teo-request} are such that
    \begin{align*}
   x \triangleright
    y =L^0_xL_x(y)= L^0_x L_{L^{0}_y(x)}(y)
    = \lambda_x\rho_{\lambda^{-}_y(x)}(y),
    \end{align*}
for all $x,y\in X$. 
Similarly, the same equality holds for quasi racks satisfying 
\eqref{ast} and \eqref{astastast}. 
This is exactly the definition of the associated rack of arbitrary left non-degenerate solution, where $\lambda^{-}_y(x) = \lambda^{-1}_y(x)$. We anticipate that, in \cref{Section3}, we will introduce shelf structures associated to quasi left non-degenerate solutions such that the above identity is always satisfied.
\end{rem}

\smallskip

\begin{cor}
    Let $(X, \triangleright)$ be a quasi rack. Then the derived map $r_\triangleright$ is a left non-degenerate solution if and only if $(X, \triangleright)$ is a rack and $r_\triangleright$ is of the form $r(x, y) = (y, y\triangleright x)$, for all $x, y \in X$.
    \begin{proof}
       As, for every $y\in X$ the map $L_y^{0}$ is idempotent, it follows that it is bijective if and only if $L^{0}_{y} = \id_X$. Then $L_{y} L_{y}^{-} = \id_X = L_{y}^{-}L_{y}$ and thus $L_{y}$ are also bijective,  for all $y\in Y$. 
    \end{proof}
\end{cor}

\smallskip

Below, we prove that \eqref{L0} is a sufficient condition for the derived map to be also quasi bijective.
\begin{theor}\label{teo_quasibij}
   Let $(X, \triangleright)$ be a quasi rack and assume that \eqref{L0} is satisfied.  
    Then the derived map $r_\triangleright$ associated to $(X, \triangleright)$
is quasi bijective and quasi non-degenerate solution.
\begin{proof}
  Let $x, y \in X$ and consider the map $r_\triangleright^-(x, y)=\left( L^-_x(y), L^0_y(x)\right)$. 
  Let us prove that $r_\triangleright r_\triangleright^-r_\triangleright=r_\triangleright, \quad r_\triangleright^-r_\triangleright r_\triangleright^-=r_\triangleright^-, \quad r_\triangleright r_\triangleright^-=r_\triangleright^- r_\triangleright$. First, we claim that
     \begin{align*}
        f(x, y):=r_\triangleright^-r_\triangleright(x, y)=r_\triangleright r_\triangleright^-(x, y)=\left(L^0_y(x), L^0_x(y)\right).
    \end{align*}
    In fact, $r_\triangleright^-r_\triangleright(x, y)=\left(L^-_{L^0_x(y)}L_y(x), \, L^0_{L_y(x)}L^0_x(y)\right)$ and
    \begin{align*}
       L^-_{L^0_x(y)}L_y(x)\underset{\eqref{request1}}{=}L^0_{L^0_x(y)}(x)\underset{\eqref{astastast}}{=}
       L^0_{L^0_x(y)}L^0_x(x)
       \underset{\ref{lemma_Lx_gene}-1.}{=}
       L^0_yL^0_x(x) \underset{\eqref{astastast}}{=} L^0_y(x)
    \end{align*}
    and
    \begin{align*}
      L^0_{L_y(x)}L^0_x(y)\underset{\eqref{L0}}{=}L^0_{L_y(x)}L^0_yL^0_x(y) \underset{\ref{lemma_Lx_gene}-3.}{=}L^0_xL^0_y(y)\underset{\eqref{L0}}{=}L^0_x(y).
    \end{align*}
    Moreover, $r_\triangleright r_\triangleright^-(x, y)=\left(L^0_{L^-_x(y)}L^0_y(x), \, L_{L^0_y(x)}L_x^-(y)\right)$ and
    \begin{align*}
        L^0_{L^-_x(y)}L^0_y(x)\underset{\eqref{L0}}{=}L^0_xL^0_{L^-_x(y)}L^0_y(x)\underset{\eqref{self_distri}}{=}L^0_{L^0_x(y)}L^0_xL^0_y(x)\underset{\ref{lemma_Lx_gene}-1.}{=}L^0_xL^0_y(x)\underset{\eqref{L0}}{=}L^0_y(x)
    \end{align*}
    and 
    \begin{align*}
        L^{}_{L^0_y(x)}L_x^-(y)
        \underset{\eqref{astastast}}{=}
        L^{}_{L^0_y(x)}L^0_yL_x^-(y)
        \underset{\ref{lemma_Lx_gene}-1.}{=}
        =L^0_y L_x L^-_x(y)
        \underset{\eqref{astastast}}{=}
        = L^0_x(y).
    \end{align*}
    Thus, we get
    \begin{align*}
        r_\triangleright r_\triangleright^-r_\triangleright(x,y)
        = fr_\triangleright(x,y)
        &=\left(L^0_{L_y(x)}L^0_x(y), \, L^0_{L^0_x(y)}L_y(x)\right)\\
        &= \left(L^0_{L_y(x)}L^0_yL^0_x(y), \, L_x^0L^0_{L^0_x(y)}L_y(x)\right) &\mbox{by \eqref{L0}}\\
        &= \left(  L^0_xL^0_y(y), L_x^0L_y(x) \right)&\mbox{by \cref{lemma_Lx_gene}-1.}\\
        &=\left(L_x^0(y), L_y(x)\right)&\mbox{by \eqref{L0}}\\
        &=r_\triangleright(x, y).
    \end{align*}
    By the previous computations,
  \begin{align*}
      r_\triangleright^-r_\triangleright r_\triangleright^-(x,y)
      =r_\triangleright^-f(x, y)
      &=\left(L^-_{L^0_y(x)}L_x^0(y), \, L^0_{L_x^0(y)}L^0_y(x)\right)\\
      &=\left(L^0_yL^-_{L^0_y(x)}L^0_x(y), \, L^0_xL^0_y(x)\right)&\mbox{by \eqref{astastast} \& \ref{lemma_Lx_gene}-1.}\\
      &= \left(L^0_yL_x^-(y), \, L^0_y(x)\right)&\mbox{by \eqref{astastast} \& by \ref{lemma_Lx_gene}-1.}\\
      &=\left(L_x^-(y), L^0_y(x)\right)&\mbox{by \eqref{L0}}\\
      &=r_{\triangleright}^-(x, y).
  \end{align*}  
  Consequently,  by \cref{r-}, $\left(X, r^-\right)$ is a solution and it is clearly quasi non-degenerate.
\end{proof}
\end{theor}

\smallskip

\begin{lemma}\label{lemma_solu_meno}
    Let $(X, \triangleright)$ be a quasi rack and assume that \eqref{L0} is satisfied, for all $x\in X$. Then, $(X, \triangleleft)$ with 
       $y\triangleleft x:= L^-_x(y)$, for all $x,y\in X$,
    is a right quasi rack. 
\end{lemma}
\begin{proof}
It directly follows by observing that, by \cref{lemma_Lx_gene}-2., the right multiplications $R_x:X\to X, y\mapsto y\triangleleft x$ satisfy the right version of \eqref{self_distri}.
\end{proof}

\begin{rem}
 Note that the map $r_\triangleright^-$ in \cref{teo_quasibij} is the solution associated to the right quasi rack $\left(X, \triangleleft \right)$.
\end{rem}

\smallskip

\begin{cor}
Let $(X, \triangleright)$ be a quasi quandle. Then the derived map $r_\triangleright$ is a quasi bijective and quasi non-degenerate solution.
    \begin{proof}
        It is enough to observe that condition \eqref{L0} holds since, for all $x \in X$, we have $L^0_x(x)=x$. Thus, by \cref{teo_quasibij}, the claim follows.
    \end{proof}
\end{cor}

\smallskip

In what follows, we provide classes of examples of quasi racks satisfying both \eqref{request0} and \eqref{request1} and, in particular, whose associated derived maps are quasi non-degenerate and quasi bijective solutions. Note that condition \eqref{L0} is not satisfied by the second class of examples we provide.

\begin{exs} Let $X$ be a Clifford semigroup.
\begin{enumerate}
    \item Let $(X, \triangleright)$ be the conjugation quasi quandle with $L_y(x)=y^-xy$, for all $x, y \in X$. Then, $L^0_x(x)=x$, for all $x \in X$. Thus, by \cref{teo_quasibij}, the map $$r_{\triangleright}(x, y)=\left( x^0y, y^-xy\right),$$
    for all $x, y \in X$, is a quasi bijective and quasi non-degenerate solution. 
    It is, indeed, the map associated with the trivial weak brace on $X$.
    \item More general, consider the quasi rack $(X, \triangleright)$ 
    given by $L_x(y)=x^{-1}yxe$, for all $x, y \in X$, where $e$ is an idempotent. Then, by \cref{teo-request}, the map
    $$r_{\triangleright}(x, y)=\left( x^0ye, y^-xye\right),$$ 
    for all $x, y \in X$, is a quasi non-degenerate solution.  In particular, this is a \emph{deformed solution} by an idempotent $e \in \E(X)$ and it is also quasi bijective (for more details, see \cite[Theorem 3.10]{MaRySt25}). In this case, $L^0_x(x) \neq x$, so this shows that condition \eqref{L0} is not necessary to ensure that the quasi bijectivity property holds.
    \item Let us consider the core quasi quandle $(X, \triangleright)$, i.e., $L_x(y)=xy^{-}x$, for all $x, y \in X$. Then, $L^0_x(x)=x$, for all $x \in X$. Thus, by \cref{teo_quasibij}, the map $$r_{\triangleright}(x, y)=\left( x^0y, yx^-y\right),$$
    for all $x, y \in X$, is a quasi bijective and quasi non-degenerate solution.
\end{enumerate}   
\end{exs}

\medskip

In \cite{JePi25}, the authors established the significant fact that non-degenerate solutions are bijective. This naturally raises the question of whether analogous results hold for quasi non-degenerate solutions, although addressing this problem presents significant challenges.
\begin{que}
    Is every quasi non-degenerate solution $(X, r)$ also quasi bijective?
\end{que}

\bigskip

To illustrate the diversity of quasi racks, we finish this section with giving the numbers, up to isomorphism, of these structures of size $\leqslant 4$. Note that in general the problem of enumerating all self-distributive structures of size $n$ is challenging even for small values of $n$, see \cite{Je97}. On the other hand, more involved methods have been used in \cite{VoYa19} to compute all racks and quandles of size up to $13$. 
Let us denote by $r(n)$ the number of racks of size $n$ (computed in \cite{VoYa19}), $qr(n)$ the number of quasi racks of size $n$ and $ds(n)$ the number of quasi racks such that the derived map from \cref{def:derived_map} is a solution, up to isomorphism. We also denote by $qr^{(*)}(n)$, $qr^{(**)}(n)$ and $qr^{(***)}(n)$, numbers of solutions, up to isomorphism, satisfying conditions \eqref{ast}, \eqref{astast} and \eqref{astastast}, respectively. The enumeration, obtained using GAP system \cite{GAP4}, appears in Table~\ref{table: enumeration}. 

\smallskip

\begin{table}[hbt!]
\begin{center}
\begin{tabular}{|c|c|c|c|c|c|c|}
\hline
$n$ & $r(n)$ & $qr(n)$ & $ds(n)$ & $qr^{(*)}(n)$ & $qr^{(**)}(n)$ & $qr^{(***)}(n)$ \\ \hline
$n= 2$            &  2   &     5     &       4        &  4 & 4  &   3 \\ \hline
$n= 3$            & 6    & 31          & 20                & 17  & 19   & 13    \\ \hline
$n= 4$            & 19    & 325         & 169               & 90  & 151  &  91   \\ \hline
\end{tabular}
\end{center}
\caption{Enumeration of quasi racks of small order}
\label{table: enumeration}
\end{table}

\medskip

\section{Shelves associated to quasi left-non degenerate solutions}\label{Section3}
In this section, we identify a family of quasi left non-degenerate solutions that give rise to a shelf structure. In this way, the well-known result on left non-degenerate solutions (namely, the converse part of \cref{YBE_and_shelf}) can be situated within a broader framework.

\smallskip

In the series of technical propositions, we introduce three conditions on quasi left non-degenerate solutions, denoted by \eqref{A}, \eqref{B} and \eqref{C}, that are sufficient to obtain the postulated derived shelf structures.

\begin{prop}\label{propA}
  Let $(X, r)$ be a quasi left non-degenerate solution. Then the following identities 
\begin{align}\label{A}
\forall\,x,y\in X\qquad\lambda^0_{\lambda_x(y)}=\lambda_x^0\lambda_y^0 \tag{A}
\end{align}  
and
$$
\forall\,x,y\in X\qquad \quad 
\lambda_{\lambda_x(y)}=\lambda_x\lambda_y\lambda_{\rho_y(x)}^-
$$
are equivalent.
    \begin{proof}
If $x,y \in X$ and condition \eqref{A} holds, then  
    \begin{align*}
     \lambda_{\lambda_x(y)}
&=\lambda_{\lambda_x(y)}\lambda^0_{\lambda_x(y)}=\lambda_{\lambda_x(y)}(\lambda_x\lambda_y)^0\underset{\eqref{first}}{=}\lambda_{\lambda_x(y)}\left(\lambda_{\lambda_x(y)}\lambda_{\rho_y(x)}\right)^0\\
&=\lambda_{\lambda_x(y)}\lambda_{\rho_y(x)}^0=\lambda_{\lambda_x(y)}\lambda_{\rho_y(x)}\lambda^-_{\rho_y(x)}\underset{\eqref{first}}{=}\lambda_x\lambda_y\lambda_{\rho_y(x)}^-.
    \end{align*}
The converse follows easily from \eqref{first}.    
\end{proof}
\end{prop}

\smallskip

\begin{prop}\label{propB}
  Let $(X, r)$ be a quasi left non-degenerate solution such that 
\begin{align}\label{B}
\rho_y(x)=\lambda^0_{\lambda_x(y)}\rho_{\lambda^0_x(y)}(x), \tag{B}
\end{align}
for all $x,y\in X$.  Then 
\begin{enumerate}
    \item[{\rm 1.}] $\rho_y(x) = \rho_{\lambda^0_x(y)}(x)$,
    \item[{\rm 2.}] $\lambda_x\lambda_y = \lambda_x\lambda_{\lambda^0_x(y)}$, \ 
    $\lambda^0_x\lambda_y = \lambda^0_x\lambda_{\lambda^0_x(y)}$,\ \ and \ \ $\lambda^-_x\lambda_y = \lambda^-_x\lambda_{\lambda^0_x(y)}$,
\end{enumerate} for all $x, y \in X$.
\begin{proof}
  If $x, y \in X$, we obtain \begin{align*}
    \rho_y(x) = \lambda^0_{\lambda_x(y)}\rho_{\lambda^0_x(y)}(x)
    = \lambda^0_{\lambda_x(\lambda^0_x(y))}\rho_{\lambda^0_x(\lambda^0_x(y))}(x) = \rho_{\lambda^0_x(y)}(x),
\end{align*}
thus
\begin{align*}
    \lambda_x\lambda_y 
    \underset{\eqref{first}}{=}
    \lambda_{\lambda_x(y)}\lambda_{\rho_y(x)}
    = \lambda_{\lambda_x(\lambda^0_x(y))}\lambda_{\rho_{\lambda^0_x(y)}(x)}
    \underset{\eqref{first}}{=}
    \lambda_x\lambda_{\lambda^0_x(y)},
\end{align*}
thus, the first identity in $2.$ holds. Moreover, applying the map $\lambda^-_x$ on the left in the last equality, we obtain the other ones.
\end{proof}
\end{prop}

\smallskip

\begin{prop}\label{propAB}
    Let $(X, r)$ be a quasi left non-degenerate solution such that conditions \eqref{A} and \eqref{B} are satisfied.
 Then the following hold:
\begin{enumerate}
    \item[{\rm 1.}] $\lambda_{\lambda^0_x(y)} = \lambda^0_x\lambda_y$ \ \ and \ \
    $\lambda^0_{\lambda^0_x(y)} = \lambda^0_x\lambda^0_y$,
    \item[{\rm 2.}] $\lambda_{\rho_y(x)}=\lambda^-_{\lambda_{x}(y)}\lambda_x\lambda_y $,
    \item[{\rm 3.}] $\lambda^-_{\lambda_x(y)}\lambda_x=\lambda_{\rho_y(x)}\lambda_y^-$,
    \item[{\rm 4.}] $\lambda_{\rho_{\lambda^0_z(y)}(x)}=\lambda_{\rho_{y}(x)}\lambda^0_{\lambda^0_z(y)}$, 
    \item[{\rm 5.}] $\lambda_{\rho_{\lambda^0_z(y)}(x)}\rho_{\lambda^-_z(y)}(z)
     =  \lambda_{\rho_{y}(x)}\rho_{\lambda^-_z(y)}(z)$,
         \item[{\rm 6.}]$\lambda^0_{\lambda^-_x(y)}\lambda^-_x = \lambda^0_{\lambda^0_x(y)}\lambda^-_x$,
\end{enumerate}
for all $x,y,z\in X$.
\end{prop}
\begin{proof}
    Let $x,y,z\in X$. Notice that  
\begin{align*}
    \lambda_{\lambda^0_x(y)}
    &= \lambda_{\lambda_x(\lambda^-_x(y))}
    \underset{\ref{propA}}{=} \lambda_x\lambda_{\lambda^-_x(y)}\lambda^-_{\rho_{\lambda^-_x(y)}(x)}
    =\lambda^0_x\lambda_x\lambda_{\lambda^-_x(y)}\lambda^-_{\rho_{\lambda^-_x(y)}(x)}\\
    &\underset{\ref{propA}}{=} \lambda^0_x \lambda_{\lambda^0_x(y)}
    \underset{\ref{propB}-2.}{=} \lambda^0_x\lambda_y,
\end{align*}
and hence 
\begin{align*}
 \lambda^0_{\lambda^0_x(y)}   
 = \lambda^-_{\lambda^0_x(y)} \lambda^{}_{\lambda^0_x(y)} 
 = \lambda^-_y\lambda_x^0\lambda_x^0\lambda_y
 = \lambda_x^0\lambda_y^0.
\end{align*}
Moreover, 
    \begin{align*}
        \lambda_{\rho_y(x)}
        =
        \lambda_{\lambda^0_{\lambda_x(y)}\rho_{\lambda^0_x(y)}(x)}
        \underset{1.}{=} 
        \lambda^0_{\lambda_x(y)}\lambda_{\rho_{\lambda^0_x(y)}(x)}
        \underset{\ref{propB}-1.}{=}     \lambda^0_{\lambda_x(y)}\lambda_{\rho_y(x)}
        \underset{\eqref{first}}{=} 
        \lambda^-_{\lambda_x(y)}\lambda_x\lambda_y.
    \end{align*}
It follows that
\begin{align*}
    \lambda^-_{\lambda_x(y)}\lambda_x=   \lambda^-_{\lambda_x(y)}\lambda^0_{\lambda_x(y)}\lambda_x=\lambda^-_{\lambda_x(y)}\lambda^0_x\lambda_y^0\lambda_x=\lambda^-_{\lambda_x(y)}\lambda_x\lambda_y\lambda_y^-=\lambda_{\rho_y(x)}\lambda_y^-.
\end{align*}
Besides, we obtain that
\begin{align*}
\lambda_{\rho_{\lambda^0_z(y)}(x)}
&= \lambda_{\lambda^0_{\lambda_x\lambda^0_z(y)}\rho_{\lambda^0_x\lambda^0_z(y)}(x)}
\underset{\ref{propB}-1.}{=} 
\lambda_{\lambda^0_{\lambda_x\lambda^0_z(y)}\rho_{\lambda^0_z(y)}(x)}
\underset{1.}{=}
\lambda^0_{\lambda_x\lambda^0_z(y)}\lambda_{\rho_{\lambda^0_z(y)}(x)}\\
&= \lambda^-_{\lambda_x\lambda^0_z(y)}\lambda_{\lambda_x\lambda^0_z(y)}\lambda_{\rho_{\lambda^0_z(y)}(x)}\underset{\eqref{first}}{=}
\lambda^-_{\lambda_x\lambda^0_z(y)}\lambda_x\lambda_{\lambda^0_z(y)}=\lambda^-_{\lambda^0_z\lambda_x(y)}\lambda_x\lambda_{\lambda^0_z(y)}\\
&\underset{1.}{=}\lambda^0_z\lambda^-_{\lambda_x(y)}\lambda_x\lambda_{\lambda^0_z(y)}
\underset{3.}{=}
\lambda_{\rho_y(x)}\lambda^-_y\lambda^0_z\lambda_{\lambda^0_z(y)}\underset{1.}{=}\lambda_{\rho_y(x)}\lambda^-_y\lambda^0_z\lambda_y
= \lambda_{\rho_y(x)}\lambda^0_z\lambda^0_y\\
&\underset{1.}{=}\lambda_{\rho_y(x)}\lambda^0_{\lambda^0_z(y)},
\end{align*}
and hence
\begin{align*}
       \lambda_{\rho_{\lambda^0_z(y)}(x)}\rho_{\lambda^-_z(y)}(z)=  \lambda_{\rho_{y}(x)}\lambda^0_{\lambda^0_z(y)}\rho_{\lambda^-_z(y)}(z)
       = \lambda_{\rho_{y}(x)}\rho_{\lambda^-_z(y)}(z).
    \end{align*}
Finally, \begin{align*}
       \lambda^0_{\lambda^-_x(y)}\lambda^-_x 
       &= \lambda_{\lambda^-_x(y)}\big(\lambda_x \lambda_{\lambda^-_x(y)}\big)^-
       \underset{\eqref{first}}{=}\lambda_{\lambda^-_x(y)}\big(
       \lambda_{\lambda^0_x(y)}
       \lambda_{\rho_{\lambda^-_x(y)}(x)}\big)^-
       \underset{3.}{=} \big(\lambda^-_{\lambda^0_x(y)}\lambda_x\big)^-\lambda^-_{\lambda^0_x(y)}\\&= \lambda^-_x\lambda^0_{\lambda^0_x(y)},
    \end{align*}
hence $6.$ holds. 
Therefore, the claim follows.

\end{proof}

\smallskip
\begin{rem}\label{conseq.Y1}
  Note that in the proofs of  
  Propositions \ref{propA}, \ref{propB} and \ref{propAB}
  only equation \eqref{first} is used.
\end{rem}

\smallskip

\begin{prop}\label{lem-lambda0-2}
     Let $(X, r)$ be a quasi left non-degenerate solution satisfying conditions \eqref{A} and \eqref{B} and such that 
     \begin{align}\label{C}
         \lambda^0_x\rho_y = \rho_y\lambda_x^0, \tag{C}
     \end{align}
     for all $x,y\in X.$ Then, the following are satisfied:
     \begin{enumerate}
         \item[{\rm 1.}] $\rho_{\lambda^0_y\lambda^-_z(x)}(z) = \lambda^0_y\rho_{\lambda^-_z(x)}(z)$, 
         \item[{\rm 2.}] $\lambda_{\rho_{\lambda^-_y(x)}(y)}\lambda^0_y = \lambda_{\rho_{\lambda^-_y(x)}(y)}$,
         \item[{\rm 3.}] $\lambda_y\lambda_{\rho_{\lambda^-_x(y)}(x)}=\lambda_x\lambda_{\lambda^-_{x}(y)}$,
     \end{enumerate}
     for all $x,y,z\in X$.
\end{prop}
\begin{proof}
    Let $x,y,z\in X$. Then
    \begin{align*}
    \rho_{\lambda^0_y\lambda^-_z(x)}(z) 
    &= \lambda^0_{\lambda^0_y\lambda^0_z(x)}\rho_{\lambda^0_y\lambda^-_z(x)}(z)
    \underset{\ref{propAB}-1.}{=} \lambda^0_y
    \lambda^0_{\lambda^0_z(x)}
    \rho_{\lambda^0_y\lambda^0_z\lambda^-_z(x)}(z)\\
    &\underset{\ref{propAB}-1.}{=}\lambda^0_y
    \lambda^0_{\lambda^0_z(x)}
    \rho_{\lambda^0_{\lambda^0_y(z)}\lambda^-_z(x)}(z)= \lambda^0_{\lambda^0_z(x)}\rho_{\lambda^0_{\lambda^0_y(z)}\lambda^-_z(x)}\lambda^0_y(z)\\
    &\underset{\ref{propB}-1.}{=} \lambda^0_{\lambda^0_z(x)} \rho_{\lambda^-_z(x)}\lambda^0_y(z)
    = \lambda^0_y \lambda^0_{\lambda^0_z(x)}\rho_{\lambda^-_z(x)}(z)      = \lambda^0_y\rho_{\lambda^-_z(x)}.
    \end{align*}
    Moreover, 
   we obtain
\begin{align*}
        \lambda_{\rho_{\lambda^-_y(x)}(y)}\lambda^0_y
        \underset{\ref{propAB}-4.}{=} \lambda_{\rho_{\lambda^-_y(x)}(y)}\lambda^0_{\lambda^-_y(x)}\lambda^0_y
        \underset{\ref{propAB}-1.}{=}
        \lambda_{\rho_{\lambda^-_y(x)}(y)}\lambda^0_{\lambda^-_y(x)}
        \underset{\ref{propAB}-4.}{=} \lambda_{\rho_{\lambda^-_y(x)}(y)}.
    \end{align*}
 From the last equality, exchanging the role of $x$ and $y$, we get
    \begin{align*}
        \lambda_y\lambda_{\rho_{\lambda^-_x(y)}(x)}&=\lambda^0_x \lambda_y\lambda_{\rho_{\lambda^-_x(y)}(x)}\underset{\ref{propAB}-1.}{=}  \lambda_{\lambda^0_x(y)}\lambda_{\rho_{\lambda_x^-(y)}(x)}
        \underset{\eqref{first}}{=}\lambda_x\lambda_{\lambda^-_{x}(y)},
    \end{align*} 
    which completes the proof.
\end{proof}

\smallskip

\begin{lemma}\label{prop-lambdahom}
       Let $(X, r)$ be a quasi left non-degenerate solution satisfying conditions \eqref{A} and \eqref{B}.  Set $x \, \triangleright_r \, y:= \lambda_x\rho_{\lambda^{-}_y(x)}(y)$,
    for all $x, y \in X$, then the following hold:
       \begin{enumerate}
       \item[{\rm 1.}] $\rho_y(x)=\lambda^-_{\lambda_x(y)}\left(\lambda_x(y) \triangleright_r x\right)$, for all $x,y \in X$,
           \item[{\rm 2.}] the maps $\lambda_x, \lambda^-_x$, and $\lambda_x^0$ are homomorphisms of the magma $(X, \triangleright_r)$.
       \end{enumerate}
\begin{proof}
Initially, let us observe that $1.$ follows directly from the condition \eqref{B}. Now, if $a, b \in X$, then
    \begin{align*}
        \lambda_x\left(a \triangleright_r b\right)&=\lambda_x\lambda_a\rho_{\lambda^-_b(a)}(b)\underset{\eqref{first}}{=}\lambda_{\lambda_x(a)}\lambda_{\rho_a(x)}\rho_{\lambda^-_b(a)}(b)
        \underset{\ref{propAB}-5.}{=}\lambda_{\lambda_x(a)}\lambda_{\rho_{\lambda^0_b(a)}(x)}\rho_{\lambda^-_b(a)}(b)\\
        &\underset{\eqref{second}}{=}\lambda_{\lambda_x(a)}\rho_{\lambda_{\rho_b(x)}\lambda_b^-(a)}\lambda_x(b)\underset{\ref{propAB}-3.}{=} \lambda_{\lambda_x(a)}\rho_{\lambda^-_{\lambda_x(b)}\lambda_x(a)}\lambda_x(b)=\lambda_x(a) \triangleright_r \lambda_x(b),
    \end{align*}
    for all $x \in X$. In addition,
\begin{align*}
    \lambda^0_x(a)\triangleright_r\lambda^0_x(b)
    &= \underbrace{\lambda^0_x\lambda_a}_{\ref{propAB}-1.}
    \rho_{\lambda^-_{\lambda^0_x(b)}\lambda^0_x(a)}\lambda_x^0(b)
    =\lambda^0_x\lambda_a
    \underbrace{\rho_{\lambda^0_{\lambda^0_x(b)}\lambda^-_b(a)}}_{\ref{propB}-2.\,\&\ \ref{propAB}-1.}\lambda_x^0(b)\\
    &=
    \lambda^0_x\lambda_a
    \rho_{\lambda^-_b(a)}(b)
    = \lambda^0_x(a\triangleright_r b),
\end{align*}
for all $x \in X$. Then, it follows that
\begin{align*}
    \lambda^-_x(a)\triangleright_r \lambda^-_x(b)
    &= \lambda_x^0\lambda^-_x(a)\triangleright_r \lambda_x^0\lambda^-_x(b)
    = \lambda_x^0\left(\lambda^-_x(a)\triangleright_r \lambda^-_x(b)\right)
    = \lambda^-_x\left(\lambda^{}_x\lambda_x^-(a)\triangleright_r \lambda^{}_x\lambda_x^-(b)\right)\\[0.2cm]
    &= \lambda^-_x\lambda^0_x\left(a\triangleright_r b\right)
    = \lambda^-_x\left(a\triangleright_r b\right),
\end{align*}
for all $x \in X$. Therefore, the assertion holds.  
\end{proof}
\end{lemma}

\smallskip
\begin{rem}\label{conseq.prop-lambdahom}
  Note that in the proof of \cref{prop-lambdahom} only 
  equation \eqref{first} is used.
\end{rem}

We are now in a position to prove the main theorem of the present section.
\smallskip
\begin{theor}\label{theo-shelf-special-lnd-sol}
   Let $(X, r)$ be a quasi left non-degenerate solution such that 
\begin{align}
\lambda^0_{\lambda_x(y)}&=\lambda_x^0\lambda_y^0,\tag{A} \label{eq:lambda-idemp}\\
\rho_y(x)&=\lambda^0_{\lambda_x(y)}\rho_{\lambda^0_x(y)}(x), \tag{B} \label{eq:rho-idemp} \\
\lambda^0_x\rho_y &= \rho_y\lambda^0_x, \tag{C}\label{eq:comm-lambda-rho}
\end{align}
for all $x,y \in X$.  Set $x \, \triangleright_r \, y:= \lambda_x\rho_{\lambda^{-}_y(x)}(y)$
    for all $x, y \in X$. Then the structure $(X, \triangleright_r)$ is a shelf.
\begin{proof}
For $x, y, z \in X$, we have
\begin{align*}
            (x\triangleright_r y)&\triangleright_r(x\triangleright_r z) = \lambda_{\lambda_x\rho_{\lambda^{-}_y(x)}(y)}
            \rho_{\lambda^{-}_{\lambda_x\rho_{\lambda^{-}_z(x)}(z)}\lambda_x\rho_{\lambda^{-}_y(x)}(y)}\lambda_x\rho_{\lambda^{-}_z(x)}(z)\\
            &=\lambda_{\lambda_x\rho_{\lambda^{-}_y(x)}(y)}\underbrace{\rho_{\lambda_{\rho_{\rho_{\lambda_z^-(x)}(z)}(x)}\lambda^-_{\rho_{\lambda^-_z(x)}(z)}\rho_{\lambda^{-}_y(x)}(y)}}_{\text{by \ref{propAB}-3.}}\lambda_x\rho_{\lambda^-_z(x)}(z)\\
            &=\lambda_{\lambda_x\rho_{\lambda^{-}_y(x)}(y)} \underbrace{\lambda_{\rho_{\lambda_{\rho_{\lambda^-_z(x)}(z)}^0\rho_{\lambda^-_y(x)}(y)}(x)}\rho_{\lambda^-_{\rho_{\lambda^-_z(x)}(z)}\rho_{\lambda^-_y(x)}(y)}\rho_{\lambda^-_z(x)}(z)}_{\text{by} \ \eqref{second}}\\
            &=\lambda_{\lambda_x\rho_{\lambda^{-}_y(x)}(y)} \underbrace{\lambda_{\rho_{\rho_{\lambda^-_y(x)}(y)}(x)}\rho_{\lambda^-_{\rho_{\lambda^-_z(x)}(z)}\rho_{\lambda^-_y(x)}(y)}\rho_{\lambda^-_z(x)}(z)}_{\text{by \ref{propAB}-5.}}\\
            &= \underbrace{\lambda_x\lambda_{\rho_{\lambda^-_y(x)}(y)}}_{\text{by \eqref{first}}}\rho_{\lambda^-_{\rho_{\lambda^-_z(x)}(z)}\rho_{\lambda^-_y(x)}(y)}\rho_{\lambda^-_z(x)}(z)
        \end{align*}       
and
        \begin{align*}
            x \triangleright_r (y \triangleright_r z)
            &=\lambda_x
            \rho_{\lambda^-_{\lambda_y\rho_{\lambda^-_z(y)}(z)}(x)}\lambda_y\rho_{\lambda^-_z(y)}(z)
            =\lambda_x
            \underbrace{\rho_{\lambda_{\rho_{\rho_{\lambda^-_z(y)}(z)}(y)}\lambda^-_{\rho_{\lambda^-_z(y)}(z)}\lambda^-_y(x)}}_{\text{by \ref{propA}}}\lambda_y\rho_{\lambda^-_z(y)}(z)
            \\
            &= \lambda_x
            \underbrace{\lambda_{\rho_{\lambda^0_{\rho_{\lambda^-_z(y)}(z)}\lambda^-_y(x)}(y)}\rho_{\lambda^-_{\rho_{\lambda^-_z(y)}(z)}\lambda^-_y(x)}\rho_{\lambda^-_z(y)}(z)
            }_{\text{by \eqref{second}}}\\
            &= \lambda_x
            \underbrace{\lambda_{\rho_{\lambda^-_y(x)}(y)}\rho_{\lambda^-_{\rho_{\lambda^-_z(y)}(z)}\lambda^-_y(x)}\rho_{\lambda^-_z(y)}(z)
            }_{\text{by \ref{propAB}-5.}}.
        \end{align*}  
        Now notice that
        \begin{align*}
            \rho&_{\lambda^-_{\rho_{\lambda^-_z(x)}(z)}\rho_{\lambda^-_y(x)}(y)}\rho_{\lambda^-_z(x)}(z)=\rho_{\underbrace{\lambda^-_{\lambda^-_z(x)}\lambda^-_z\lambda_{\lambda^0_z(x)}}_{\text{by \ref{propAB}-2.}}\rho_{\lambda^-_y(x)}(y)}\rho_{\lambda^-_z(x)}(z)\\
            &=\rho_{\lambda^-_{\lambda^-_z(x)}\underbrace{\lambda^-_z\lambda_{x}}_{\text{by \ref{propB}-2.}}\rho_{\lambda^-_y(x)}(y)}\rho_{\lambda^-_z(x)}(z)=\rho_{\lambda^-_{\lambda^-_z(x)}\lambda^-_z\left(x \, \triangleright_r \,  y\right)}\rho_{\lambda^-_z(x)}(z)\\
            &=\rho_{\lambda^-_{\lambda^-_z(x)}\underbrace{\left(\lambda^-_z(x) \, \triangleright_r \,  \lambda^-_z(y)\right)}_{\text{by \ref{prop-lambdahom}-2.}}}\rho_{\lambda^-_z(x)}(z)  =\rho_{\lambda^0_{\lambda^-_z(x)}\rho_{\lambda^-_{\lambda^-_z(y)}\lambda^-_z(x)}\lambda^-_z(y)}\rho_{\lambda^-_z(x)}(z)\\
            &=\rho_{\underbrace{\rho_{\lambda^-_{\lambda^-_z(y)}\lambda^-_z(x)}\lambda^0_{\lambda^0_z(x)}\lambda^-_z(y)}_{\text{by \ref{propAB}-6.}}}\rho_{\lambda^-_z(x)}(z)=\rho_{\rho_{\lambda^-_{\lambda^-_z(y)}\lambda^-_z(x)}\lambda^-_z(y)}\underbrace{\lambda^0_{\lambda^0_z(x)}\rho_{\lambda^-_z(x)}(z)}_{\text{by \ref{lem-lambda0-2}-1.}}\\
            &=\rho_{\rho_{\lambda^-_{\lambda^-_z(y)}\lambda^-_z(x)}\lambda^-_z(y)}\rho_{\lambda^-_z(x)}(z).
        \end{align*}
        On the other hand, 
        \begin{align*}
            &\rho_{\lambda^-_{\rho_{\lambda^-_z(y)}(z)}\lambda^-_y(x)}\rho_{\lambda^-_z(y)}(z)\underset{\eqref{third}}{=}\rho_{\rho_{\lambda^-_{\rho_{\lambda^-_z(y)}(z)}\lambda^-_y(x)}\lambda^-_z(y)}\rho_{\lambda_{\lambda^-_z(y)}\lambda^-_{\rho_{\lambda^-_z(y)}(z)}\lambda^-_y(x)}   (z)      \\
            &\underset{\ref{lem-lambda0-2}-3.}{=}\rho_{\rho_{\lambda^-_{\lambda^-_z(y)}\lambda^-_z(x)}\lambda^-_z(y)}\rho_{ \lambda^0_{\lambda^-_z(y)}\lambda^-_z(x)}(z) 
        \ \underset{\ref{propAB}-6.}{=}\rho_{\rho_{\lambda^-_{\lambda^-_z(y)}\lambda^-_z(x)}\lambda^-_z(y)}\rho_{\lambda^0_{\lambda^0_z(y)}\lambda^-_z(x)}(z) \\
            &\underset{\ref{propAB}-1.}{=}\rho_{\rho_{\lambda^-_{\lambda^-_z(y)}\lambda^-_z(x)}\lambda^-_z(y)}\rho_{\lambda^0_y\lambda^-_z(x)}(z)
            \underset{\ref{lem-lambda0-2}-1.}{=}
            \lambda^0_y\rho_{\rho_{\lambda^-_{\lambda^-_z(y)}\lambda^-_z(x)}\lambda^-_z(y)}\rho_{\lambda^-_z(x)}(z).
        \end{align*}
        Hence,
\begin{align*}
    x \triangleright_r (y \triangleright_r z)
            &=\lambda_x
            \lambda_{\rho_{\lambda^-_y(x)}(y)}\lambda^0_y
            \rho_{\lambda^-_{\rho^{}_{\lambda^-_z(x)}(z)}\rho_{\lambda^-_y(x)}(y)}
            \rho^{}_{\lambda^-_z(x)}(z)\\
&\underset{\ref{lem-lambda0-2}-2.}{=}
\lambda_x\lambda_{\rho_{\lambda^-_y(x)}(y)}
            \rho_{\lambda^-_{\rho^{}_{\lambda^-_z(x)}(z)}\rho_{\lambda^-_y(x)}(y)}
            \rho^{}_{\lambda^-_z(x)}(z)
            = (x\triangleright_r y)\triangleright_r(x\triangleright_r z).
\end{align*} 
Therefore the claim follows.

\end{proof}\end{theor}

\smallskip

Note that \eqref{A}, \eqref{B}, and \eqref{C} are sufficient conditions for obtaining a shelf, but as the following example shows, they are not necessary. Namely, there exists a quasi left non-degenerate solution $(X, r)$ that does not satisfy \eqref{B}, but $\left(X, \triangleright_r\right)$ is still a shelf.

\begin{ex}
Let us consider the solution over the set $X = \{1, 2, 3\}$ as in \cref{ex:quasi-Lyubashenko}, with $f(1) = 1$, $f(2) = f(3) =2$ and $g(1) = g(2) = 2$, $g(3) = 3$. Then the condition \eqref{B} is not satisfied, but $(X, \triangleright_{r})$ such that $x \, \triangleright_r \, y = \lambda_x\rho_{\lambda^{-}_y(x)}(y) = fg(y)$ is a quasi rack.    
\end{ex}

\smallskip

As a direct application of the theorem, we obtain the following result.

\smallskip

\begin{cor}
Let $(X, r)$ be a quasi left non-degenerate solution of the form $r(x, y) = (\lambda(y), \rho_{y}(x))$ for certain maps $\lambda$ and $\rho_{y}$, satisfying the following conditions
\begin{align*}
    \lambda^{0} \rho_{x} = \rho_{x}\lambda^{0}
    \qquad\text{and}\qquad
    \rho_{x} = \lambda^{0}\rho_{\lambda^{0}(x)}
\end{align*}
for all $x\in X$.
Then the map $ s(x, y) = (\lambda^{0}(y), \lambda\rho_{\lambda^{-}(y)}(x))$ is a quasi left non-degenerate solution.
\end{cor}
\begin{proof} Let us denote $y\triangleright_r x = \lambda\rho_{\lambda^{-}(y)}(x)$, for all $x, y\in X$.  By Theorem~\ref{theo-shelf-special-lnd-sol}, $\left(X, \triangleright_r\right)$ is a shelf. Then the map $s(x, y) = (L^{0}(y), y\triangleright_r x)$ is a solution if and only if the following equations hold
\begin{align*}
    \label{morphism}&\lambda^{0}(z \triangleright_r y) = \lambda^{0}(z) \triangleright_r \lambda^{0}(y) \tag{Y2}\\
    \label{twisted-shelf}&x \triangleright_r (y \triangleright_r z) = (x \triangleright_r y) \triangleright_r (\lambda^{0}(x) \triangleright_r z) \tag{Y3}
\end{align*}
for all $x,y,z\in X$. The equality \ref{morphism} follows from \cref{prop-lambdahom}.
Moreover, if $x, y, z\in X$
$$
\lambda^{0}(x) \triangleright_r z = \lambda\rho_{\lambda^{-}\lambda^{0}(x)}(z) = \lambda\rho_{\lambda^{-}(x)}(z) = x\triangleright_r z,
$$ 
so condition \ref{twisted-shelf} follows from Theorem~\ref{theo-shelf-special-lnd-sol}. 
\end{proof}

\smallskip

\begin{ex}\label{prop_L_dual}
    Let $(S, r_S)$ be the solution associated with a dual weak brace $\left(S, +, \circ\right)$. Then conditions \eqref{A}, \eqref{B}, and \eqref{C} are satisfied. 
    Thus $\left(S, \triangleright_r\right)$ is a quasi rack that coincides with the conjugation quasi quandle on $S$. In fact, if $x,y\in S$, then
\begin{align*}
    x\triangleright_r y &= \lambda_x\rho_{\lambda_{y^-}(x)}(y)
    = -x + x\circ\rho_{\lambda_{y^-}(x)}(y)
    = -x + x\circ\left(y^- + \lambda_{y^-}\left(x\right)\right)^-\circ \lambda_{y^-}\left(x\right)\\
    &= -x + x\circ\left(y^-\circ x\right)^-\circ \lambda_{y^-}\left(x\right)
    = -x + x^0 + y\circ y^-\circ\left(y+x\right)
    = -x + y^0 + y+ x\\
    &= -x + y + x,
\end{align*}  
that is our claim.

\end{ex}

\smallskip

\cref{bijective=rack} states that a left non-degenerate solution $(X, r)$ is bijective if and only if the associated shelf $\left(X, \triangleright_r\right)$ is a rack. Moreover, the above example suggests that the following natural generalization of this fact holds.
\begin{que}
Is a shelf $(X, \triangleright_r)$ associated to any left quasi non-degenerate and quasi bijective solution satisfying \eqref{A}, \eqref{B} and \eqref{C} a quasi rack?
\end{que}

\medskip

\section{A description of quasi left non-degenerate solutions}\label{sec:twists}
In this section, we introduce the notion of \emph{g-twist} that generalizes the known one of \emph{Drinfel'd twist} in the context of non-degenerate bijective solutions, see \cite{Do21, Dr85}. Consistently with \cite{DoRySt24}, quasi left non-degenerate solutions as in \cref{theo-shelf-special-lnd-sol} can be described in terms of g-twists.

\medskip

\begin{lemma}\label{lemma_hom_varphi}
    Let $(X, \triangleright)$ be a shelf and consider a map $\varphi: X \to \Hom(X, \triangleright), \, a \mapsto \varphi_a$ such that each $\varphi_a$ admits a {\rm(}unique{\rm)} relative inverse and $\varphi_a^0L_b=L_b\varphi_a^0$, for all $a, b \in X$. Then $\varphi_a^0$ and $\varphi_a^-$ are shelf homomorphisms, for all $a \in X$.
    \begin{proof}
        First, if $a, b \in X$,  then 
       \begin{align*}
          L_{\varphi^-_a(b)}\varphi^-_a&=\varphi_a^0L_{\varphi^-_a(b)}\varphi^-_a\\
            &=\varphi_a^-L_{\varphi^0_a(b)}\varphi^0_a &\mbox{since $\varphi_a \in \Hom(X, \triangleright)$}\\
            &=\varphi_a^-L_{\varphi^0_a(b)}. 
        \end{align*}
        Hence, it follows
        \begin{align*}
           L_{ \varphi_a^0(b)}\varphi_a^0=L_{\varphi_a^-\varphi_a(b)}\varphi_a^-\varphi_a=\varphi_a^-L_{\varphi^0_a\varphi_a(b)}\varphi_a=\varphi_a^-L_{\varphi_a(b)}\varphi_a{=}\varphi^0_aL_b,
        \end{align*}
where in the last equality again we use $\varphi_a \in \Hom(X, \triangleright)$.
    Proceeding like in the last part of the proof of \cref{prop-lambdahom}, one can check that $\varphi_a^-$ also is a shelf homomorphism.
    \end{proof}
\end{lemma}

\smallskip

\begin{defin}\label{def:generalized_twist}
    Let $(X, \triangleright)$ be a shelf. A map $\varphi: X \to \Hom(X, \triangleright), \, a \mapsto \varphi_a$ such that each $\varphi_a$ admits a (unique) relative commuting inverse and
 \begin{align}
\varphi_a^0L_b=L_b\varphi_a^0, \ \quad\text{and}\quad   \varphi^0_{\varphi_a(b)}=\varphi^0_a\varphi^0_b, \quad\text{and}\quad \varphi^0_{\varphi^0_a(b)}L_{\varphi^0_a(b)}(a)= L_{b}(a) \label{L0-com}
 \end{align}   
 hold, for all $a, b \in X$, is said to be a \emph{g-twist} if the following hold
\begin{align}\label{eq:varphi}
       &\varphi_a\varphi_b = \varphi_{\varphi_a(b)}\varphi_{\varphi^{-}_{\varphi_a(b)}L_{\varphi_a(b)}(a)} 
    \end{align}
for all $a, b \in X$.
\end{defin}



\smallskip

The computations appearing in the proof of the following result are analogous to those of \cite[Theorem 2.15]{DoRySt24}. Nevertheless, for completeness, we include them in our framework.
\begin{theor}\label{le:lndsol}
    Let $\left(X,\,\triangleright\right)$ be a left shelf and $\varphi:X
\to X^X,$ $a\mapsto \varphi_a$ a map such that each $\varphi_a$ admits a {\rm(}unique{\rm)} relative commuting inverse and conditions \eqref{L0-com} are satisfied. Then, the map 
$r_\varphi:X\times X\to X\times X$ defined by
    \begin{equation}\label{prop:solform}
        r_\varphi\left(a, b\right)  
        = \left(\varphi_a\left(b\right),\, \varphi^{-}_{\varphi_a\left(b\right)}\left(\varphi_a\left(b\right)\triangleright a\right)\right),  
    \end{equation}
    for all $a,b\in X$ is a quasi left non-degenerate solution as in \cref{theo-shelf-special-lnd-sol} if and only if $\varphi$ is a g-twist. Moreover, 
    any quasi left non-degenerate solution as in \cref{theo-shelf-special-lnd-sol} can be obtained that way.
\end{theor}
\begin{proof}
    First, let us assume that $\varphi$ is a g-twist and set $\rho_{b}\left(a\right):= \varphi^{-}_{\varphi_a\left(b\right)}L_{\varphi_a\left(b\right)}\left(a\right)$, for all $a,b\in X$. Then \eqref{first} follows by \eqref{eq:varphi}. Now, by \cref{conseq.Y1}, proceeding as in the proof of \cref{propAB}-3., since the last two equalities in \eqref{L0-com} correspond to \eqref{A} and \eqref{B}, we have
    \begin{align}
\label{eq_y1} \forall\, a, b \in X \quad \varphi^-_{\varphi_a(b)}\varphi_a=\varphi_{\rho_b(a)}\varphi_b^-.
    \end{align}
    Thus, if $a,b,c\in X$, we get 
    \begin{align*}
     \rho_{\varphi_{\rho_b\left(a\right)}\left(c\right)}\varphi_a\left(b\right) 
    &= \varphi^{-}_{\varphi_{\varphi_a\left(b\right)}\varphi_{\rho_b\left(a\right)}\left(c\right)}
    L_{\varphi_{\varphi_a\left(b\right)}\varphi_{\rho_b\left(a\right)}\left(c\right)}\varphi_a\left(b\right)\\
    &=\varphi^{-}_{\varphi_a\varphi_b\left(c\right)}L_{\varphi_a\varphi_b\left(c\right)}{\varphi_a\left(b\right)}&\mbox{by \eqref{first}}\\
    &= \varphi^{-}_{\varphi_a\varphi_b\left(c\right)}\varphi_aL_{\varphi_b\left(c\right)}\left(b\right)
    &\mbox{$\varphi_a\in \Hom\left(X,\,\triangleright\right)$}\\
    &=\varphi_{\rho_{\varphi_b\left(c\right)}\left(a\right)}\varphi^{-}_{\varphi_b\left(c\right)}L_{\varphi_b\left(c\right)
    }\left(b\right)&\mbox{by \eqref{eq_y1}}\\
    &= \varphi_{\rho_{\varphi_b\left(c\right)}\left(a\right)}\rho_c\left(b\right),
    \end{align*}
hence \eqref{second} holds.
In particular, from the previous equalities,  we obtain that
\begin{equation}\label{eq:2}
    \varphi_{\rho_{\varphi_b\left(c\right)}\left(a\right)}\rho_c\left(b\right) =
 \varphi^{-}_{\varphi_a\varphi_b\left(c\right)}\varphi_aL_{\varphi_b\left(c\right)}\left(b\right).
\end{equation}
Moreover, by applying twice \eqref{first}, 
\begin{equation}\label{eq:2.2}
    \varphi_{\varphi_a\varphi_b\left(c\right)}\varphi_{\rho_{\varphi_{\rho_b\left(a\right)}\left(c\right)}\varphi_a\left(b\right)}
    = \varphi_{\varphi_{\varphi_a\left(b\right)}\varphi_{\rho_{b}\left(a\right)}\left(c\right)}
\varphi_{\rho_{\varphi_{\rho_b\left(a\right)}\left(c\right)}\varphi_a\left(b\right)}
    = \varphi_{\varphi_a\left(b\right)}\varphi_{\varphi_{\rho_b\left(a\right)}\left(c\right)}.
\end{equation}
Noting that, by \cref{lemma_hom_varphi}, $\varphi^0_a, \varphi^{-}_a\in \Hom(X,\triangleright)$, for all $a\in X$, we obtain that
$$
\begin{aligned}
&\rho_{\rho_c\left(b\right)}\rho_{\varphi_b\left(c\right)}\left(a\right)=
 \varphi^{-}_{\varphi_{\rho_{\varphi_b\left(c\right)}\left(a\right)}\rho_c\left(b\right)}
L_{\varphi_{\rho_{\varphi_b\left(c\right)}\left(a\right)}\rho_c\left(b\right)}\rho_{\varphi_b\left(c\right)}\left(a\right)
\\
&= \varphi^{-}_{\rho_{\varphi_{\rho_b\left(a\right)}\left(c\right)}\varphi_a\left(b\right)}
L_{\varphi^{-}_{\varphi_a\varphi_b\left(c\right)}\varphi_aL_{\varphi_b\left(c\right)}\left(b\right)}
\varphi^{-}_{\varphi_a\varphi_b\left(c\right)}
L_{\varphi_a\varphi_b\left(c\right)}\left(a\right)
&\mbox{by \eqref{second} and \eqref{eq:2}}\\
&= \varphi^{-}_{\rho_{\varphi_{\rho_b\left(a\right)}\left(c\right)}\varphi_a\left(b\right)}\varphi^{-}_{\varphi_a\varphi_b\left(c\right)}
L_{L_{\varphi_a\varphi_b\left(c\right)}\varphi_a\left(b\right)}L_{\varphi_a\varphi_b\left(c\right)}\left(a\right)
&\mbox{$\varphi_a, \varphi^{-}_{\varphi_a\varphi_b\left(c\right)}\in \Hom\left(X,\,\triangleright\right)$}\\
&= \varphi^{-}_{\varphi_{\rho_b\left(a\right)}\left(c\right)}
\varphi^{-}_{\varphi_a\left(b\right)}
L_{\varphi_a\varphi_b\left(c\right)}L_{\varphi_a\left(b\right)}\left(a\right)
&\mbox{by \eqref{self_distri}\ \text{and} \ \eqref{eq:2.2}}\\
&= \varphi^{-}_{\varphi_{\rho_b\left(a\right)}\left(c\right)}
    L_{\varphi^{-}_{\varphi_a\left(b\right)}\varphi_a\varphi_b\left(c\right)}
    \varphi^{-}_{\varphi_a\left(b\right)}L_{\varphi_a\left(b\right)}\left(a\right)
    &\mbox{$\varphi^{-}_{\varphi_a(b)}\in \Hom\left(X,\,\triangleright\right)$}\\
&= \varphi^{-1}_{\varphi_{\rho_b\left(a\right)}\left(c\right)}
    L_{\varphi^0_{\varphi_a(b)}\varphi_{\rho_b(a)}(c)}
\varphi^0_{\varphi_a(b)}\varphi^{-}_{\varphi_a\left(b\right)}L_{\varphi_a\left(b\right)}\left(a\right) &\mbox{by \eqref{first}}\\
&=\varphi^{-}_{\varphi_{\rho_b\left(a\right)}\left(c\right)}
    \varphi^0_{\varphi_a(b)}L_{\varphi_{\rho_b(a)}(c)}\varphi^{-}_{\varphi_a\left(b\right)}L_{\varphi_a\left(b\right)}\left(a\right)&\mbox{$\varphi^0_{\varphi_a(b)} \in \Hom(X, \triangleright)$}\\
&=\varphi^{-}_{\varphi_{\rho_b\left(a\right)}\left(c\right)}L_{\varphi_{\rho_b(a)}(c)}
\varphi^{-}_{\varphi_a\left(b\right)}L_{\varphi_a\left(b\right)}\left(a\right) \\
&= \varphi^{-}_{\varphi_{\rho_b\left(a\right)}\left(c\right)}
    L_{\varphi_{\rho_b\left(a\right)}\left(c\right)}\rho_b\left(a\right)
    \\ 
    &=\rho_{c}\rho_{b}\left(a\right),
    \end{aligned}
$$
i.e., \eqref{third} is satisfied. Hence, $r_\varphi$ is a quasi left non-degenerate solution satisfying \eqref{A} and \eqref{B}.  Moreover, by \cref{conseq.Y1}, using \cref{propAB}-1., since
\begin{align*}
 \varphi^0_c\rho_b(a)&= \varphi^0_c \varphi^-_{\varphi_a(b)}L_{\varphi_a(b)}(a)=\varphi^0_c \varphi^-_{\varphi_a(b)}\varphi^0_cL_{\varphi_a(b)}(a)=\varphi^-_{\varphi^0_c\varphi_a(b)}L_{\varphi^0_c\varphi_a(b)}\varphi^0_c(a)\\
  &=\varphi^-_{\varphi_{\varphi^0_c(a)}(b)}L_{\varphi^0_c\varphi_a(b)}\varphi^0_c(a)=\rho_b\varphi^0_c(a),
\end{align*}
it follows that also \eqref{C} is satisfied. In addition, since, for all $a, b \in X$, $\varphi^-_{\varphi^0_b(a)}=\varphi_a^0\varphi_b$ and $\varphi^0_{\varphi^0_a(b)}=\varphi_a^0\varphi_b^0$ (as in \cref{propAB}-1.),
\begin{align*}
    a \triangleright_r b=\varphi_a\varphi^-_{\varphi^0_b(a)}L_{\varphi^0_b(a)}(b)=\varphi_a^0\varphi_b^0L_{\varphi^0_b(a)}(b)=\varphi^0_{\varphi^0_a(b)}L_{\varphi^0_b(a)}(b)= L_{a}(b).
\end{align*}
Conversely, assume that the map $r_\varphi$ is a solution on the set $X$ as in \cref{theo-shelf-special-lnd-sol}. Then, \eqref{first} 
coincides with the identity \eqref{eq:varphi}. 
Furthermore, by applying \eqref{first} to the identity  \eqref{second} and looking at the previous computations, we get
\begin{align}\label{eq:Y1+Y2}
\varphi^{-}_{\varphi_a\varphi_b\left(c\right)}\varphi_a  L_{\varphi_b\left(c\right)}\left(b\right) = 
\varphi^{-}_{\varphi_a\varphi_b\left(c\right)}L_{\varphi_a\varphi_b\left(c\right)}\varphi_a\left(b\right), 
\end{align}
for all $a,b,c\in X$. 
Moreover, 
if $x, y \in X$, we have
\begin{align*}
   \varphi_a L_{y}(x)&= \varphi_a \varphi^0_{\varphi^0_x(y)}L_{\varphi^0_x(y)}(x) \underset{\ref{propAB}-1.}{=} \varphi_a^0 \varphi_x^0 \varphi_y^0\varphi^{}_aL_{\varphi^0_x(y)}(x)=\varphi^0_{\varphi_a \varphi^0_x(y)}\varphi_a
    L_{ \varphi^{0}_x(y)}(x)\\
    &= \varphi^{}_{\varphi_a \varphi^{0}_x(y)}
    \varphi^-_{\varphi_a \varphi^{0}_x(y)}
    \varphi^{}_a L_{\varphi^0_x(y)}(x)\underset{\eqref{eq:Y1+Y2}}{=}
    \varphi^{}_{\varphi_a \varphi^{0}_x(y)}
    \varphi^-_{\varphi_a \varphi^{0}_x(y)}
     L_{\varphi^{}_a\varphi^0_x(y)}\varphi^{}_a(x)\\
     &\underset{\ref{propAB}-1.}{=}\varphi^0_{\varphi^0_{\varphi_a(x)}\varphi_a(y)}
     L_{\varphi^0_{\varphi_a(x)}\varphi_a(y)}\varphi_a(x)
    = L_{\varphi_a(y)}\varphi_a(x)
\end{align*}
%
i.e., $\varphi_{a}\in \Hom\left(X,\, \triangleright\right)$, for every $a\in X$.
Therefore, $\varphi$ is a g-twist.

To observe that any quasi left non-degenerate solution $r(a,b)= \left(\lambda_a\left(b\right), \rho_b\left(a\right)\right)$ as in \cref{theo-shelf-special-lnd-sol}
can be obtained that way, it suffices to show that $\lambda$ is a g-twist of the shelf $(X,\triangleright_r),$ where  $a\triangleright_r b= \lambda_a\rho_{\lambda^{-}_b\left(a\right)}\left(b\right).$In fact, by \cref{prop-lambdahom}, $\lambda_a$ is a shelf homomorphism.
Clearly, $\rho_b\left(a\right) = \lambda^{-}_{\lambda_a\left(b\right)}\left(\lambda_a\left(b\right)\triangleright_r a\right)$, and the map $r$ can be written 
as in \eqref{prop:solform}. One can easily show that the identities \eqref{L0-com} and \eqref{eq:varphi} are satisfied. Therefore the statement is proven.
\end{proof}

\smallskip

\begin{rem}
 If $(X, \triangleright)$ is a quasi rack that satisfies  conditions \eqref{request0} and \eqref{request1} then, for all $x \in X$, the map $\varphi_x=L_x$ determines a  g-twist. Indeed, by the definition $L^0_xL_y=L_yL^0_x$ and $L_x \in \Hom(X, \triangleright)$. In addition, by \eqref{request0} $L^0_{L_x(y)}=L^0_xL^0_y$ and by \eqref{request1} $L^0_{L^0_x(y)}L_{L^0_x(y)}(x)= L_{y}(x)$, for all $x, y \in X$.
 Moreover,
\begin{align*}
\varphi_{\varphi_x\left(y\right)}\varphi_{\varphi^{-1}_{\varphi_x\left(y\right)}L_{\varphi_x\left(y\right)}\left(x\right)}
&= L_{L_x\left(y\right)}L_{L^{0}_{L_x\left(y\right)}\left(x\right)} = L_{L^{}_{L_x\left(y\right)}L^0_{L_x\left(y\right)}(x)}L_{L_x\left(y\right)}\\
&= L_{L_{L_x\left(y\right)}(x)}L_{L_x\left(y\right)}
= L_{L_x(y)}L_x = L_xL_y
= \varphi_x\varphi_y,
\end{align*}
for all $x,y\in X$. 
In particular, we obtain \begin{align*}
    r_\varphi\left(x, y\right)  
        = \big(\varphi_x\left(y\right),\, \varphi^{-1}_{\varphi_x\left(y\right)}\left(\varphi_x\left(y\right)\triangleright x\right)\big)
        = \big(L_x\left(y\right),\, L^{0}_{L_x\left(y\right)}(x)\big)
\end{align*}
which is exactly the derived map.\\
Let us notice that quasi racks from \cref{ex_idemp} and \cref{exs_quasi} are such quasi racks.
\end{rem}

\medskip

\section{P{\l}onka sum of racks}\label{sec:Plonka}

In this section, we focus on the notion of the  P{\l}onka sum of racks, according to the general notion of \emph{sum of a direct system of algebras} (or \emph{sum of a semilattice ordered system of algebras}) introduced in \cite{Plo67, Plo67-2, PloRom92}. 

\smallskip

\begin{defin}  
Let $Y$ be a (lower) semilattice and $\{\left(X_\alpha, \triangleright_\alpha\right) \mid \alpha\in Y\}$ a family of disjoint racks indexed by $Y$. For all $\alpha,\beta\in Y$ such that $\alpha\geq \beta$, let $\phi_{\alpha,\beta}:X_{\alpha}\to X_{\beta}$ be a rack homomorphism such that 
    \begin{enumerate}
     \item $\phi_{\alpha,\alpha}=\id_{X_{\alpha}}$, for every $\alpha \in Y$;
     \item $\phi_{\beta,\gamma}\phi_{\alpha, \beta} = \phi_{\alpha, \gamma}$, for all $\alpha,\beta,\gamma \in Y$ such that $\alpha\geq\beta\geq\gamma$.
     \end{enumerate} 
Set $X:=\displaystyle \bigcup_{\alpha \in Y} X_\alpha$. Then, the operation 
\begin{align}\label{triangle}
    \forall\, a\in X_\alpha, \ \forall\,  b\in X_\beta\qquad
    a\triangleright b:= \phi_{\alpha,\alpha\beta}(a)\triangleright_{\alpha\beta} \phi_{\beta,\alpha\beta}(b)
    \end{align}
endows $X$ with a structure of a shelf that we call \emph{P{\l}onka sum of the racks $\left(X_\alpha, \triangleright_\alpha\right)$ indexed by $Y$}.
\end{defin}

\smallskip

Note that if $X$ is a P{\l}onka sum of racks indexed by a certain $Y$, then the \emph{projection map} $\pi:X\to Y$ defined by $\pi(a) = \alpha$, if $a\in X_\alpha$ is a semigroup epimorphism. We underline that this happens in general for P{\l}onka sum of algebras. 

Our aim is to describe a class of quasi racks that are P{\l}onka sum of racks. As will be seen, these are quasi racks that give rise to a derived solution.

\smallskip

\begin{theor}\label{theo-Plonka_sums}
Let $\left(X, \triangleright\right)$ be a P{\l}onka sum of racks $\left(X_\alpha, \triangleright_\alpha\right)$ indexed by a {\rm(}lower{\rm)} semilattice $Y$. Then $\left(X, \triangleright\right)$ is a quasi rack such that the following identities 
\begin{align}\label{eq:L0&associative}
   L^0_a(a) = a
   \qquad\text{and}\qquad
   L^{0}_{L_a(b)} = L^0_aL^0_b,
\end{align}
are satisfied, for all  $a,b\in X$. Moreover, the derived map $r_\triangleright$ in \eqref{r_triang} associated to $\left(X, \triangleright \right)$ is a quasi bijective and quasi non-degenerate solution. 
\begin{proof}
    For convenience, for $a \in X_\alpha, b \in X_\beta$, we set
    \begin{align*}
        L^{[\alpha\beta]}_a(b):=a \triangleright b.
    \end{align*}
    Following the \cref{def:quasi-rack}, we show that, for all $\alpha, \beta \in Y$, $a \in X_\alpha$, $b \in X_\beta$,
    \begin{align*}
        L_a^-(b)=\left(L^{[\alpha\beta]}_{\phi_{\alpha, \alpha\beta}(a)}\right)^{-1}\phi_{\beta, \alpha\beta}(b).
    \end{align*}
    Let us compute 
    \begin{align*}L^-_aL_a(b)&=L_a^-\left(L^{[\alpha\beta]}_{\phi_{\alpha,\alpha\beta}(a)}\phi_{\beta,\alpha\beta}(b)\right)=\left(L^{[\alpha\alpha\beta]}_{\phi_{\alpha, \alpha\alpha\beta(a)}}\right)^{-1}\phi_{\alpha\beta, \alpha\alpha\beta}L^{[\alpha\beta]}_{\phi_{\alpha,\alpha\beta}(a)}\phi_{\beta,\alpha\beta}(b)\\
    &=\left(L^{[\alpha\beta]}_{\phi_{\alpha, \alpha\beta(a)}}\right)^{-1}\phi_{\alpha\beta, \alpha\beta}L^{[\alpha\beta]}_{\phi_{\alpha,\alpha\beta}(a)}\phi_{\beta,\alpha\beta}(b)\\
    &=\phi_{\beta, \alpha\beta}(b).
    \end{align*}
    Similarly, $L_aL_a^-(b)=\phi_{\beta, \alpha\beta}(b)$. Thus, set $L^0_a(b):=\phi_{\beta, \alpha\beta}(b)$, one has that
    \begin{align*}
        L_aL_a^-L_a(b)=L_aL_a^0(b)=L_a\left(\phi_{\beta, \alpha\beta}(b)\right)=L^{[\alpha\alpha\beta]}_{\phi_{\alpha, \alpha\alpha\beta}(a)}\phi_{\alpha\beta, \alpha\alpha\beta}\phi_{\beta, \alpha\beta}(b)=L_a(b)
    \end{align*}
    and, analogously, $L_a^-L_a^0(b)=L_a^-(b)$. In addition, if $x \in X_\gamma$,
    \begin{align*}
L_a^0L_x(b)=L_a^0\left(L^{[\beta\gamma]}_{\phi_{\gamma, \gamma\beta(x)}}\phi_{\beta, \beta\gamma}(b)\right)=\phi_{\beta\gamma, \alpha\beta\gamma}L^{[\beta\gamma]}_{\phi_{\gamma, \gamma\beta(x)}}\phi_{\beta, \beta\gamma}(b)=L^{[\alpha\beta\gamma]}_{\phi_{\gamma, \alpha\beta\gamma}(x)}\phi_{\beta, \alpha\beta\gamma}(b)
    \end{align*}
    and with similar computations $L_xL^0_a(b)=L^{[\alpha\beta\gamma]}_{\phi_{\gamma, \alpha\beta\gamma}(x)}\phi_{\beta, \alpha\beta\gamma}(b)$.\\
Finally, if $a \in X_\alpha$, $b \in X_\beta$, and $c \in X_{\gamma}$,
\begin{align*}
    L^0_aL^0_b(c) 
    = \phi_{\beta\gamma, \alpha\beta\gamma}\phi_{\gamma, \beta\gamma}(c)  
    = \phi_{\gamma, \alpha\beta\gamma}(c)
    = L^0_{L^{[\alpha\beta]}_{\phi_{\alpha, \alpha \beta}(a)}\phi_{\beta, \alpha\beta}(b)}(c)
    = L^0_{L_a(b)}(c).
\end{align*}
Moreover, 
$L_a^0(a)=\phi_{\alpha, \alpha}(a)=\id_{X_\alpha}(a)=a$. Therefore, 
by \cref{teo_quasibij}, the map $r_\triangleright$ is a quasi bijective and quasi non-degenerate solution. 
\end{proof}
\end{theor}

\smallskip

\begin{rem}
    Notice that, if each $(X_\alpha, \triangleright_\alpha)$ in \cref{theo-Plonka_sums} is a quandle, then $(X,\triangleright)$ is a quasi quandle.
\end{rem}

\smallskip

\begin{theor}
    Let $\left(X, \triangleright\right)$ be a quasi rack satisfying the equalities in \eqref{eq:L0&associative}. Then $\left(X, \triangleright\right)$ is a P{\l}onka sum of racks.
\end{theor}
\begin{proof}
    Let $Y:=\{L_a^0 \ | \ a\in X\}$ and $\pi:X\to Y$ the surjection defined by $\pi(a) = L_a^{0}$, for all $a\in X$. It is easy to check that the second identity in \eqref{eq:L0&associative} and the properties of the maps $L_a^0$ of the quasi rack $X$ make $Y$ a semilattice. 
    Moreover, again by the second identity in \eqref{eq:L0&associative}, we obtain that $\pi$ is a homomorphism from $\left(X, \triangleright\right)$ into $Y$. 
    Let $\sim_{\pi}$ be the equivalence relation induced by $\pi$ and define $X_{\pi(a)}:= [a]_{\sim_\pi}$, for all $a \in X$. In addition, for all $\pi(a)\geq \pi(b)$, let $\phi_{\pi(a), \pi(b)}:X_{\pi(a)}\to X_{\pi(b)}$ be the map defined by 
    $$
    \phi_{\pi(a), \pi(b)}(x) = \pi(b)(x)=L^0_b(x),
    $$
    for all  $x\in X_{\pi(a)}$. Moreover, in the following, let us denote by $\triangleright_{\pi(a)}$ the restriction of $\triangleright$ to $X_{\pi(a)}$, for all $\pi(a) \in Y$.\\
    First, we show that $\left(X_{\pi(a)}, \triangleright_{\pi(a)}\right)$ is a rack, for all $\pi(a) \in Y$.
    In fact, if $x, y \in X_{\pi(a)}$, then
    \begin{center}
        $\pi\left(x \triangleright_{\pi(a)} y\right)=L^0_{L_x(y)}=L^0_xL^0_y=L^0_a=\pi(a)$.
    \end{center} 
Hence, $x \triangleright_{\pi(a)} y \in X_{\pi(a)}$. Furthermore, we show that the map $L_x: X_{\pi(a)} \to X_{\pi(a)}$ is bijective, for all $x \in X_{\pi(a)}$. Observe that, if $x, y \in\pi(a)$,  
\begin{align*}
    L^0_{L^-_x(y)}
    = L^0_{L_a(L_a^-L_x^-(y))}
    \underset{\eqref{eq:L0&associative}}{=}
L^0_aL^0_{L^-_aL^-_x(y)}\underset{\ref{lemma_Lx_gene}-4.}{=}L^0_aL^0_{L^-_x(y)}=L^0_xL^0_{L^-_x(y)}\underset{\ref{lemma_Lx_gene}-4.}{=}L^0_xL^0_y=L^0_a,
\end{align*}
which means that $L^-_x(y)\in \pi(a)$. Hence,  $L_xL_x^-(y) = L_x^-L_x(y) = L^0_x(y) = L^0_y(y) =y$. \\
In addition, if $\pi(a) \geq \pi(b) \in Y$ and $x, y \in X_{\pi(a)}$, then
\begin{align*}
    \phi_{\pi(a), \pi(b)}(x \triangleright_{\pi(a)} y)=L^0_bL_x(y)\underset{\ref{lemma_Lx_gene}}{=}L_{L^0_b(x)}L^0_b(y)=\phi_{\pi(a), \pi(b)}(x) \triangleright_{\pi(b)} \phi_{\pi(a), \pi(b)}(y).
\end{align*}
Finally, if $\pi(a), \pi(b) \in Y$ and $x \in X_{\pi(a)}, y \in X_{\pi(b)}$, 
\begin{align*}
   x \triangleright y&= x \triangleright L^0_y(y)=L_xL^0_aL^0_b(y)\underset{\eqref{eq:L0&associative}}{=}L_xL^0_{L_a(b)}(y)\\
   &\underset{\ref{lemma_Lx_gene}}{=} L_{L^0_{L_a(b)}(x)}L^0_{L_a(b)}(y)\\
   &=L^0_{L_a(b)}(x) \triangleright L^0_{L_a(b)}(y)\\
   &=\phi_{\pi(a), \pi(a \ \triangleright \ b)}(x) \triangleright_{\pi(a \ \triangleright\  b)} \phi_{\pi(b), \pi(a \ \triangleright \ b)}(y)\\
   &=\phi_{\pi(a), \pi(a) \wedge \pi(b)}(x) \triangleright_{\pi(a) \wedge \pi(b} \phi_{\pi(b), \pi(a) \wedge \pi(b)}(y).
\end{align*}
Therefore, the claim follows.
\end{proof}

\smallskip

\begin{ex}
Let $Y$ be a (lower) semilattice and $\{\left(X_\alpha, \triangleright_\alpha\right) \, | \, \alpha\in Y\}$ a family of disjoint groups indexed by $Y$. Consider on each group $X_\alpha$ the conjugation (resp. core) quandle $\left(X_\alpha, \triangleright_\alpha\right)$. For all $\alpha,\beta\in Y$ such that $\alpha\geq \beta$, let $\phi_{\alpha,\beta}:X_{\alpha}\to X_{\beta}$ be groups homomorphism such that $\phi_{\alpha,\alpha}=\id_{X_{\alpha}}$, for every $\alpha \in Y$, and $\phi_{\beta,\gamma}\phi_{\alpha, \beta} = \phi_{\alpha, \gamma}$, for all $\alpha,\beta,\gamma \in Y$ such that $\alpha\geq\beta\geq\gamma$.
Then, set $X:=\displaystyle \bigcup_{\alpha \in Y} X_\alpha$, the pair $(X, \triangleright)$, with $\triangleright$ the operation \eqref{triangle} is a quasi rack. In particular, $X$ is a strong semilattice of groups, equivalently, a Clifford semigroup, and $(X, \triangleright)$ is the conjugation (resp. core) quasi quandle in \cref{exs_quasi}. The vice versa also holds by simply applying the definitions. 
\end{ex}

\smallskip

\begin{rem}
    Note that the quasi racks in~\cref{ex_idemp} and~\cref{exs_quasi}-2 cannot be a P{\l}onka sum of racks. Indeed, neither of them satisfies the condition~\eqref{L0}.
    Additionally, let us notice that observation made in 3. of~\cref{rem-ex-astastast-noast} ensures that there exist quasi racks that satisfy~\eqref{astastast} but not~\eqref{ast}.
\end{rem}

\smallskip

To conclude this section, we present the following theorem, which provides a description of the solutions arising from P{\l}onka sums of racks. It makes use of the same technique employed for the construction of quasi racks. In fact, \cite[Theorem 4.1]{CaCoSt21} provides a method for obtaining a new solution as \emph{strong semilattice} of known solutions.
\begin{theor}
    Let $\left(X, \triangleright\right)$ be a P{\l}onka sum of racks $\left( X_\alpha, \triangleright_\alpha \right)$ indexed by a {\rm(}lower{\rm)} semilattice $Y$. Then the solution $r_\triangleright$ in \eqref{r_triang} associated to $\left(X, \triangleright \right)$ is the strong semilattice of the derived solutions $\left(X_\alpha, r_\alpha \right)$ associated to each rack $X_\alpha$. In particular,
    \begin{align*}
        r_\triangleright(a, b)=r_{\alpha\beta}\left( \phi_{\alpha, \alpha\beta}(a), \phi_{\beta, \alpha\beta}(b)\right),
    \end{align*}
    for all $a \in X_\alpha$, $b \in X_\beta$.
\end{theor}
    \begin{proof}
       It is a consequence of  \cite[Theorem 4.1]{CaCoSt21}. In fact, by this result, one only has to check that $\left( \phi_{\alpha, \beta} \times \phi_{\alpha, \beta}   \right)r_\alpha=r_\beta\left( \phi_{\alpha, \beta} \times \phi_{\alpha, \beta}   \right)$, for all $\alpha, \beta \in Y$ such that $\alpha \geq \beta$. It is easily seen that this follows from the definition of derived solution and from the fact that $\phi$ is a rack homomorphism.      
    \end{proof}

\bigskip

\section*{Acknowledgements}
\noindent This work was partially supported by the University of Salento - Department of Mathematics and Physics ``E. De Giorgi”. M. Mazzotta is supported by AGRI@INTESA - ``National Centre for HPC, Big Data and Quantum Computing", CUP: F83C22000740001.\\
M. Wiertel is supported by the Polish National Agency for Academic Exchange within Bekker Programme BPN/BEK/2024/1/00311/U/00001.\\
M. Mazzotta and P. Stefanelli are members of GNSAGA (INdAM) and the nonprofit association ``AGTA-Advances in Group Theory and Applications". 



\bigskip
\bigskip
\bibliography{bibliography}

\end{document}